\documentclass[11pt, reqno]{amsart}
\usepackage[ruled]{algorithm2e}
\usepackage{amssymb}
\usepackage{bm}
\usepackage{booktabs, makecell, tabularx, xltabular}
\usepackage{enumitem}
\usepackage[mathscr]{eucal}
\usepackage{fullpage}
\usepackage[draft=false]{hyperref}
\usepackage[noabbrev]{cleveref}
\usepackage{mathtools}
\usepackage{microtype}
\usepackage{scalerel}
\usepackage{stmaryrd}
\SetSymbolFont{stmry}{bold}{U}{stmry}{m}{n}
\usepackage{todonotes}
\usepackage{xparse}

\usepackage[format=plain, labelfont={bf}]{caption}
\usepackage{float}
\usepackage{graphicx}
\graphicspath{{./experiments/figures/}}
\usepackage[labelformat=simple]{subcaption}

\allowdisplaybreaks

\numberwithin{equation}{section}

    \renewcommand{\epsilon}{\varepsilon}
    \renewcommand{\phi}{\varphi}

    \renewcommand{\tilde}{\widetilde}
    \renewcommand{\hat}{\widehat}

    \newcommand{\defeq}{\vcentcolon=}
    \newcommand{\eqdef}{=\vcentcolon}

    \newcommand{\set}[1]{\{ #1 \}}

    \newcommand{\R}{\mathbb{R}}

    \newcommand{\abs}[1]{\lvert #1 \rvert}

    \newcommand{\inner}[2]{\langle #1, #2 \rangle}

    \newcommand{\norm}[1]{\lVert #1 \rVert}

    \newcommand{\bigO}[2][]{\mathcal{O}_{#1}(#2)}

    \let\vec\undefined
    \NewDocumentCommand{\vec}{ o O{} m }
    { \IfValueTF {#1}
        { \bm{\mathrm{#3}}^{(#1)}_{#2} }
        { \bm{\mathrm{#3}} }
    }
    \NewDocumentCommand{\mat}{ o O{} m }
    { \IfValueTF {#1}
        { \bm{\mathrm{#3}}^{(#1)}_{#2} }
        { \bm{\mathrm{#3}} }
    }
    \NewDocumentCommand{\ten}{ o O{} m }
    { \IfValueTF {#1}
        { \bm{\mathscr{#3}}^{(#1)}_{#2} }
        { \bm{\mathscr{#3}} }
    }

    \newcommand{\otr}{\circ}
    
    \newcommand{\kron}{\otimes}
    
    \newcommand{\hd}{\ast}
    \DeclareMathOperator*{\bighd}{\scalerel*{\ast}{\bigodot}}
    \newcommand{\kr}{\odot}
    \newcommand{\bigkr}{\bigodot}

    \newcommand{\tp}{\top}
    \newcommand{\pinv}{+}

    \NewDocumentCommand{\tuck}{ o m }
    { \IfValueTF {#1}
        { \llbracket #1\,;\,#2 \rrbracket }
        { \llbracket #2 \rrbracket }
    }
    
    \newcommand{\D}{\ten{D}}

    \newcommand{\argmin}{\operatorname{argmin}}
    \newcommand{\diag}{\operatorname{diag}}
    \newcommand{\rank}{\operatorname{rank}}

\newtheorem{theorem}{Theorem}[section]
\newtheorem{proposition}[theorem]{Proposition}
\newtheorem{lemma}[theorem]{Lemma}
\newtheorem{corollary}[theorem]{Corollary}

\theoremstyle{definition}
\newtheorem{definition}[theorem]{Definition}

\theoremstyle{remark}
\newtheorem{remark}[theorem]{Remark}
\newtheorem*{note}{Note}
\newtheorem*{notation}{Notation}

\begin{document}

\title[Convergence of CP-AltLS]{Convergence of the alternating least squares algorithm for CP tensor decompositions}

\author[N.~Hu]{Nicholas Hu}
\address{Department of Mathematics, UCLA}
\email{njhu@math.ucla.edu}

\author[M.~A.~Iwen]{Mark A. Iwen}
\address{Department of Mathematics, and Department of Computational Mathematics, Science, and Engineering, MSU}
\email{iwenmark@math.msu.edu}

\author[D.~Needell]{Deanna Needell}
\address{Department of Mathematics, UCLA}
\email{deanna@math.ucla.edu}

\author[R.~Wang]{Rongrong Wang}
\address{Department of Computational Mathematics, Science, and Engineering, MSU}
\email{wangron6@math.msu.edu}

\begin{abstract}
    The alternating least squares (ALS/AltLS) method is a widely used algorithm for computing the CP decomposition of a tensor. However, its convergence theory is still incompletely understood. 
    
    In this paper, we prove explicit quantitative local convergence theorems for CP-AltLS applied to orthogonally decomposable and incoherently decomposable tensors. Specifically, we show that CP-AltLS converges polynomially with order $N-1$ for $N$\textsuperscript{th}-order orthogonally decomposable tensors and linearly for incoherently decomposable tensors, with convergence being measured in terms of the angles between the factors of the exact tensor and those of the approximate tensor. Unlike existing results, our analysis is both quantitative and constructive, applying to standard CP-AltLS and accommodating factor matrices with small but nonzero mutual coherence, while remaining applicable to tensors of arbitrary rank.

    We also confirm these rates of convergence numerically and investigate accelerating the convergence of CP-AltLS using an SVD-based coherence reduction scheme.
\end{abstract}

\maketitle

\section{Introduction}

\emph{Tensors} are multidimensional arrays that are widely used in the sciences to represent and analyze data. The multiple dimensions (or ``modes'') of a tensor can correspond to different components of the data, such as position, time, frequency, intensity, object, or type.

\emph{Tensor decompositions} are structured representations of tensors -- usually approximate in practice -- that render them more amenable to storage, manipulation, and/or analysis. 
One of the most fundamental decompositions, the \emph{CP decomposition}, expresses a tensor as a sum of simpler \emph{rank-one component} tensors which can reveal patterns or features of the underlying data. Each rank-one component is in turn the outer product of \emph{factor} vectors, and for each mode of the tensor, the factors of each component together constitute a \emph{factor matrix}. 
For instance, one of the tensors that we experiment with in \Cref{SS:exp-orth} consists of fluorescence measurements from chemical samples; each component tensor corresponds to an amino acid present in the samples \cite{Bro}. Other applications include psychometrics \cite{CarollChang}, phonetics \cite{Harshman}, and neuroscience \cite{Mocks}.

The traditional ``workhorse'' algorithm \cite{KoldaBader} for computing a CP decomposition of a tensor is the \emph{alternating least squares \textup(AltLS\textup)}\footnotemark{} method proposed by Carroll and Chang \cite{CarollChang} and Harshman \cite{Harshman} in 1970. It cycles (or ``alternates'') through each mode of the tensor and finds the optimal factor matrix for that mode, in a least squares sense, while fixing the factor matrices for all other modes. This cycle repeats until a given convergence criterion is satisfied.

\footnotetext{We choose to use the abbreviation ``AltLS'' instead of the standard ``ALS'' because of the latter's association with a neurodegenerative disorder.}

Despite the popularity of this method, its convergence theory is still incompletely understood.  Most existing results are only applicable in limited cases (discussed below), such as when the tensor consists of a single rank-one component (a ``rank-one'' tensor), and are often non-explicit, qualitative, or otherwise imprecise about the rate of convergence. In contrast, we prove explicit and quantitative convergence theorems for broad classes of tensors herein.

\subsection{Prior work}

The CP-AltLS method can be viewed as a nonlinear block Gauss--Seidel iteration for minimizing the (squared) error of the ``reconstructed'' tensor, that is, the tensor formed by recombining the factors produced by the algorithm. Using this framework, Uschmajew \cite{Uschmajew} proved that it converges locally linearly to local minima of the error satisfying a nondegeneracy condition, and also established a quantitative sufficient condition for this nondegeneracy in the case of rank-one tensors. Although applicable to general tensors, this significant result still left some important questions unaddressed. For example, convergence was proved for factor matrices initialized sufficiently close to those at a local minimum, but this closeness was unquantified.

Another line of analysis formulates CP-AltLS as a high-order power method (HOPM) \cite{HuLi,HuYe,WangChu,WangChuYu}. For instance, Wang and Chu \cite{WangChu} employed this strategy to prove almost everywhere convergence to local minima. This improved upon Uschmajew's result in that their result was global and unconditional; however, it was also qualitative and limited to the rank-one case. Subsequently, they along with Yu \cite{WangChuYu} found that ``semiorthogonality'' (a condition similar to our orthogonal decomposability condition in \Cref{D:odeco}) guaranteed the existence of a best low-rank approximation to a given tensor. They further devised a modified HOPM using the polar decomposition to maintain this orthogonality, for which they proved almost everywhere convergence to local minima of the error as in their previous paper. While semiorthogonality allowed them to consider the general-rank case, this result was also qualitative, and their analysis of this modified HOPM ultimately did not apply to the standard CP-AltLS method.

\subsection{Contributions}

In this paper, we prove explicit quantitative local convergence theorems (that is, using constructive proofs with numerical constants) for CP-AltLS applied to both \emph{orthogonally decomposable} and \emph{incoherently decomposable} tensors (see \Cref{D:odeco,D:ideco}).  
Specifically, we show that CP-AltLS converges polynomially with order $N-1$ for $N$\textsuperscript{th}-order orthogonally decomposable tensors and linearly for incoherently decomposable tensors, with convergence being measured in terms of the angles between the factors of the exact tensor and those of the approximate tensor. Our approach is direct and less technical than the aforementioned approaches and allows for factor matrices with small but nonzero mutual coherence while remaining applicable to tensors of arbitrary rank.

\subsection{Organization}

We begin in \Cref{S:notation} by defining the notation we use throughout the paper, most of which is standard in the literature. In \Cref{SS:cp}, we give a precise definition of the CP decomposition and the types of tensors that we consider: orthogonally and incoherently decomposable tensors.

This is followed by the local convergence analysis in \Cref{S:convergence}. We first analyze a variant (\Cref{CP-AltLS}) of the standard CP-AltLS algorithm (\Cref{CP-SAltLS}). The main results are \Cref{T:odeco} in \Cref{SS:odeco}, which establishes local convergence for orthogonally decomposable tensors, and \Cref{T:ideco} in \Cref{SS:ideco}, which does the same for incoherently decomposable tensors. We then explain how to extend these results to the standard algorithm in \Cref{SS:convergence-saltls} and illustrate how randomization of initial guesses yields global convergence with high probability in \Cref{SS:convergence-global}.

In \Cref{SS:exp-odeco,SS:exp-ideco}, we perform numerical experiments that validate \Cref{T:odeco,T:ideco}. We also observe convergence of the ``weights'' in the CP decomposition (see \Cref{SS:cp} and \Cref{C:weights}) in \Cref{SS:exp-weights}. Finally, in \Cref{SS:exp-orth} we investigate accelerating the convergence of CP-AltLS using an SVD-based coherence reduction scheme inspired by Sharan and Valiant's Orth-ALS and Hybrid-ALS algorithms \cite{SharanValiant}.

Technical lemmas employed in the proofs of the convergence theorems are collected in \Cref{appendix}.

\section{Notation and preliminaries} \label{S:notation}

Vectors (first-order tensors) will be represented by bold lowercase letters (e.g., $\vec{a}$);
matrices (second-order tensors) will be represented by bold uppercase letters (e.g., $\mat{A}$);
higher-order tensors will be represented by bold uppercase script letters (e.g., $\ten{X}$).

The $(i_1, \dots, i_N)$-element of a tensor $\ten{X}$ will be denoted $x_{i_1, \dots, i_N}$
or simply $x_{i_1 \cdots i_N}$ when the indices are clearly distinguishable. The
$j$\textsuperscript{th} column of a matrix $\mat{A}$ will be denoted $\vec{a}_j$.

\subsection{Tensor operations}

\begin{definition}[Inner product and norm of tensors]
    The \emph{\textup(Frobenius\textup) inner product} of tensors $\ten{X}, \ten{Y} \in
    \R^{I_1 \times \cdots \times I_N}$ is
    $\inner{\ten{X}}{\ten{Y}} \defeq \textstyle \sum_{i_1 = 1}^{I_1} \cdots \sum_{i_N = 1}^{I_N} x_{i_1 \cdots i_N} y_{i_1 \cdots i_N}$
    and the \emph{\textup(Frobenius\textup) norm} of $\ten{X}$ is $\norm{\ten{X}} \defeq
    \inner{\ten{X}}{\ten{X}}^\frac{1}{2}$.
\end{definition}

\begin{definition}[$(p, q)$-norm of a matrix]
    Let $p, q \in [1, \infty]$. The \emph{$(p, q)$-norm} of a matrix $\mat{A}$ is
    $\norm{\mat{A}}_{p, q} \defeq \sup_{\norm{\vec{x}}_p = 1} \norm{\mat{A}
    \vec{x}}_q$.
    The $(p, p)$-norm of $\mat{A}$ will simply be called the \emph{$p$-norm} of $\mat{A}$ and will be
    denoted $\norm{\mat{A}}_p$.
    The $(1, \infty)$-norm will simply be called the \emph{maximum-norm} of $\mat{A}$ and will be denoted
    $\norm{\mat{A}}_\mathrm{max}$ (because it is equal to the maximum absolute value of
    the elements of $\mat{A}$).
\end{definition}

We will also employ several different tensor/matrix products, whose definitions are
summarized in the following table.

\begin{table}[H]
\begin{tabularx}{\textwidth}{lX}
    \toprule
    \textbf{Product} & \textbf{Definition} \\
    \midrule
    \makecell[tl]{{Outer} \\ $\otr : \R^{I_1 \times \cdots \times I_N} \times \R^{J_1 \times \cdots \times J_M} 
    \to \R^{I_1 \times \cdots \times I_N \times J_1 \times \cdots \times J_M}$} 
    & \makecell[tl]{$(\ten{X} \otr \ten{Y})_{i_1 \cdots i_N j_1 \cdots j_M}$ \\ $= x_{i_1 \cdots i_N} y_{j_1 \cdots j_M}$} \\
    \midrule
    \makecell[tl]{{Kronecker} \\ $\kron : \R^{I \times J} \times \R^{K \times L} \to \R^{IK \times JL}$} 
    & \makecell[tl]{$(\mat{A} \kron \mat{B})_{(i-1)K + k, (j-1)L + \ell}$ \\ $= a_{ij} b_{k\ell}$} \\
    \midrule
    \makecell[tl]{{Hadamard} \\ $\hd : \R^{I \times J} \times \R^{I \times J} \to \R^{I \times J}$} 
    & \makecell[tl]{$(\mat{A} \hd \mat{B})_{ij}$ \\ $= a_{ij} b_{ij}$} \\
    \midrule
    \makecell[tl]{{Khatri--Rao} \\ $\kr : \R^{I \times K} \times \R^{J \times K} \to \R^{IJ \times K}$} 
    & \makecell[tl]{$(\mat{A} \kr \mat{B})_{(i-1)J + j, k}$ \\ $= a_{ik} b_{jk}$} \\
    \bottomrule
\end{tabularx}
\end{table}

Finally, we introduce notation for ``reshaping'' a tensor into a matrix and for considering
the diagonal and off-diagonal entries of a matrix.

\begin{definition}[Fibres and matricization]
    Let $\ten{X} \in \R^{I_1 \times \cdots \times I_N}$. The \emph{mode-$n$ fibres} of
    $\ten{X}$ are the vectors
    \[
        \vec{x}_{i_1, \dots, i_{n-1}, :, i_{n+1}, \dots, i_N}
        \defeq
        \begin{bmatrix}
            x_{i_1, \dots, i_{n-1}, 1, i_{n+1}, \dots, i_N} \\
            x_{i_1, \dots, i_{n-1}, 2, i_{n+1}, \dots, i_N} \\
            \vdots \\
            x_{i_1, \dots, i_{n-1}, I_n, i_{n+1}, \dots, i_N}
        \end{bmatrix}
        \in \R^{I_n}
    \]
    for $i_m \in [I_m]$ and $m \neq n$. 
    The \emph{mode-$n$ matricization} of $\ten{X}$ is the matrix $\mat{X}_{(n)} \in
    \R^{I_n \times \prod_{m \neq n} I_m}$ given by
    \[
        \mat{X}_{(n)}
        =
        \begin{bmatrix}
            \vec{x}_{1, \dots, 1, :, 1, \dots, 1}
            &
            \vec{x}_{2, \dots, 1, :, 1, \dots, 1}
            & \cdots &
            \vec{x}_{I_1, \dots, I_{n-1}, :, I_{n+1}, \dots, I_N}
        \end{bmatrix},
    \]
    where the indices of the fibres are colexicographically ordered.
\end{definition}

\begin{definition}[Diagonal and off-diagonal parts of a matrix]
    The \emph{diagonal part} of a matrix $\mat{A} \in \R^{n \times n}$ is $\D(\mat{A})
    \defeq \mat{A} \hd \mat{I}$ and its \emph{off-diagonal part} is $\D'(\mat{A}) \defeq
    \mat{A} - \D(\mat{A})$.
\end{definition}

\subsection{Tensor CP decompositions} \label{SS:cp}

In this paper, we consider algorithms for computing CP decompositions, which express a
tensor as a weighted sum of simpler (``rank-one'') tensors.

\begin{definition}[CP decomposition and rank of a tensor]
    A \emph{CP decomposition} of a tensor $\ten{X}
    \in \R^{I_1 \times \cdots \times I_N}$ is a decomposition of the form
    \[
        \ten{X}
        =
        \tuck[\vec{\lambda}]{\mat[1]{A}, \dots, \mat[N]{A}}
        \defeq
        \sum_{r=1}^R \lambda_r \vec[1][r]{a} \otr \cdots \otr \vec[N][r]{a},
    \]
    where $\vec{\lambda} \in \R^R$ and $\mat[n]{A} \in \R^{I_n \times R}$; the $\lambda_r$
    are called the \emph{weights} and the $\mat[n]{A}$ are called the \emph{factor matrices}. If $R$
    is minimal, the decomposition is called a \emph{rank decomposition} and we say that
    $\ten{X}$ has \emph{rank} $R$.

    (The double bracket operator in this definition is known as the \emph{Kruskal operator}
    and is a special case of the \emph{Tucker operator}.)
\end{definition}

For certain tensors -- namely, orthogonally decomposable tensors -- we can prove that these
algorithms converge rapidly.

\begin{definition}[Orthogonally decomposable tensor] \label{D:odeco}
    A tensor is said to be \emph{orthogonally decomposable} if it admits a CP decomposition
    whose factor matrices have orthonormal columns.
\end{definition}

\begin{remark}
    If $\ten{X} \in \R^{I_1 \times \cdots \times I_N}$ is orthogonally decomposable, then
    $\rank(\ten{X}) \leq \min_{n \in [N]} I_n$.
\end{remark}

Even if a tensor is not orthogonally decomposable, if it admits a CP decomposition whose
factor matrices have ``almost orthogonal'' columns, we will be able to establish local convergence
(albeit at a slower rate).

\begin{definition}[Coherence]
    The \emph{coherence} of a matrix $\mat{A} \in \R^{n \times n}$ with columns of unit
    2-norm is $\mu(\mat{A}) \defeq \max_{i \neq j} \abs{\inner{\vec{a}_i}{\vec{a}_j}}$.
    (This normalization ensures that $\mu(\mat{A}) \in [0, 1]$.)
\end{definition}

\begin{definition}[Incoherently decomposable tensor] \label{D:ideco}
    A tensor is said to be \emph{incoherently decomposable} if it admits a CP decomposition
    whose factor matrices have sufficiently low coherence (the columns of the factor matrices are assumed to
    be normalized).  More precisely, it is said to be \emph{$\mu$-coherently decomposable}
    if the coherence of each factor matrix is at most $\mu$.
\end{definition}

\section{Local convergence of CP-AltLS} \label{S:convergence}

We begin by reviewing the alternating least squares algorithm for CP decomposition
(CP-AltLS) \cite{KoldaBader}. 
If $\ten{X} = \tuck[\vec{\lambda}]{\mat[1]{A}, \dots, \mat[N]{A}}$, we have
\begin{align*}
    \mat{X}_{(1)} &= \mat[1]{A} \mat{\Lambda} (\mat[N]{A} \kr \mat[N-1]{A} \kr \cdots \kr
    \mat[2]{A})^\tp, \\
               &\shortvdotswithin{=}
    \mat{X}_{(n)} &= \mat[n]{A} \mat{\Lambda} (\mat[N]{A} \kr \cdots \kr \mat[n+1]{A} \kr
    \mat[n-1]{A} \kr \cdots \kr \mat[1]{A})^\tp, \\
               &\shortvdotswithin{=}
    \mat{X}_{(N)} &= \mat[N]{A} \mat{\Lambda} (\mat[N-1]{A} \kr \cdots \kr
    \mat[2]{A} \kr \mat[1]{A})^\tp,
\end{align*}
where $\mat{\Lambda} = \diag(\vec{\lambda})$. Given $\ten{X}$ and initial guesses for the
factor matrices $\mat[n]{A}$, the CP-AltLS algorithm approximates each factor matrix
by regarding the equation for each mode as a separate least squares problem for the
corresponding factor matrix, in which all other factor matrices are fixed. That is,
\begin{gather*}
    \mat[n]{A} 
    =
    \underset{\mat{A}}{\argmin{}}
    \norm{\mat{X}_{(n)} - \mat{A} \mat{\Lambda} (\mat[n]{K})^\tp}, \\
    \text{where} \quad
    \mat[n]{K}
    \defeq
    \mat[N]{A} \kr \cdots \kr \mat[n+1]{A} \kr \mat[n-1]{A} \kr \cdots \kr \mat[1]{A}.
\end{gather*}
Typically, we solve for $\mat{\hat{A}} \defeq \mat{A} \mat{\Lambda}$ and then normalize its
columns to obtain $\mat[n]{A}$, updating $\vec{\lambda}$ accordingly. The solution is then
\begin{gather*}
    \mat{\hat{A}}
    = \mat{X}_{(n)} \left((\mat[n]{K})^\tp\right)^\pinv
    = \mat{X}_{(n)} \mat[n]{K} (\mat[n]{H})^\pinv, \\
    \text{where} \quad
    \mat[n]{H}
    \defeq
    \mat[N]{G} \hd \cdots \hd \mat[n+1]{G} \hd 
    \mat[n-1]{G} \hd \cdots \hd \mat[1]{G} \\
    \text{and} \quad
    \mat[n]{G}
    \defeq
    (\mat[n]{A})^\tp \mat[n]{A}.
\end{gather*}
In each iteration, the algorithm cycles (or ``alternates'') through modes $n = 1, 2, \dots,
N$, resulting in \Cref{CP-SAltLS}.

\begin{algorithm}[htbp]
    \caption{CP-AltLS} \label{CP-SAltLS}

    \DontPrintSemicolon
    \SetAlgoVlined
    \SetKwFor{For}{for}{}{}
    \SetKwFor{While}{while}{}{}

    \KwIn{tensor $\ten{X} \in \R^{I_1 \times \cdots \times I_N}$, initial approximate factor matrices $\mat[n]{A} \in
    \R^{I_n \times R}$ with normalized columns for $n = 1, \dots, N$}
    \KwOut{approximate weights $\vec{\lambda} \in \R^R$, approximate factor matrices $\mat[n]{A} \in
    \R^{I_n \times R}$ with normalized columns for $n = 1, \dots, N$}

    \For{$n = 1$ \KwTo $N$}{
        $\mat[n]{G} \leftarrow (\mat[n]{A})^\tp \mat[n]{A}$\;
    }

    \While{stopping condition has not been satisfied}%
    {
        \For{$n = 1$ \KwTo $N$}{
            $\mat[n]{K} \leftarrow \mat[N]{A} \kr \cdots \kr \mat[n+1]{A} \kr \mat[n-1]{A} \kr \cdots \kr \mat[1]{A}$\;
            $\mat[n]{H} \leftarrow \mat[N]{G} \hd \cdots \hd \mat[n+1]{G} \hd \mat[n-1]{G} \hd \cdots \hd \mat[1]{G}$\;
            $\mat[n]{M} \leftarrow \mat{X}_{(n)} \mat[n]{K}$\;
            $\mat[n]{A} \leftarrow \mat[n]{M} (\mat[n]{H})^\pinv$\;
            normalize columns of $\mat[n]{A}$, updating $\vec{\lambda}$ accordingly\;
            $\mat[n]{G} \leftarrow (\mat[n]{A})^\tp \mat[n]{A}$\;
        }
    }
\end{algorithm}

To simplify our arguments, we will consider a variant (\Cref{CP-AltLS}) of this algorithm in which the least squares problems for a given iteration depend only on the factor matrices computed in the previous iteration (as opposed to factor matrices computed in the current iteration as well). In addition to being more amenable to analysis, this formulation has the potential computational advantage that the loops over the modes are parallelizable. For this reason, when it is necessary to distinguish between the standard algorithm and the variant, we will refer to the former as \emph{serial} CP-AltLS (CP-SAltLS) and the latter as (\emph{parallel}) CP-AltLS (CP-(P)AltLS).

\begin{algorithm}[htbp]
    \caption{CP-AltLS} \label{CP-AltLS}

    \DontPrintSemicolon
    \SetAlgoVlined
    \SetKwFor{For}{for}{}{}
    \SetKwFor{While}{while}{}{}

    \KwIn{tensor $\ten{X} \in \R^{I_1 \times \cdots \times I_N}$, initial approximate factor matrices $\mat[n]{A} \in
    \R^{I_n \times R}$ with normalized columns for $n = 1, \dots, N$}
    \KwOut{approximate weights $\vec{\lambda} \in \R^R$, approximate factor matrices $\mat[n]{A} \in
    \R^{I_n \times R}$ with normalized columns for $n = 1, \dots, N$}

    \While{stopping condition has not been satisfied}%
    {
        \For{$n = 1$ \KwTo $N$}{
            $\mat[n]{G} \leftarrow (\mat[n]{A})^\tp \mat[n]{A}$\;
            $\mat[n]{K} \leftarrow \mat[N]{A} \kr \cdots \kr \mat[n+1]{A} \kr \mat[n-1]{A} \kr \cdots \kr \mat[1]{A}$\;
        }
        \For{$n = 1$ \KwTo $N$}{
            $\mat[n]{H} \leftarrow \mat[N]{G} \hd \cdots \hd \mat[n+1]{G} \hd \mat[n-1]{G} \hd \cdots \hd \mat[1]{G}$\;
            $\mat[n]{M} \leftarrow \mat{X}_{(n)} \mat[n]{K}$\;
            $\mat[n]{A} \leftarrow \mat[n]{M} (\mat[n]{H})^\pinv$\;
            normalize columns of $\mat[n]{A}$, updating $\vec{\lambda}$ accordingly\;
        }
    }
\end{algorithm}

\begin{notation}
    \hfill
    \begin{itemize}
        \item
            The letters G, K, H, and M were chosen for the matrices above as mnemonics for
            \emph{Gram} matrix, \emph{Khatri--Rao} product, \emph{Hadamard} product, and
            \emph{matricized-tensor-times-Khatri--Rao} product (MTTKRP), respectively.

        \item
            Hadamard products of the form $\mat[N]{G} \hd \cdots \hd \mat[n+1]{G} \hd
            \mat[n-1]{G} \hd \cdots \hd \mat[1]{G}$ will be abbreviated as $\bighd_{m \neq
            n} \mat[m]{G}$.

        \item
            Khatri--Rao products of the form $\mat[N]{A} \kr \cdots \kr \mat[n+1]{A} \kr
            \mat[n-1]{A} \kr \cdots \kr \mat[1]{A}$ will be abbreviated as $\bigkr_{m \neq
            n} \mat[m]{A}$; note that the index is \emph{decreasing} in this product.

        \item
            Iteration numbers will be appended to the superscripts of vectors and matrices
            produced by the algorithm (e.g., $\vec[k]{\lambda}$ will denote $\vec{\lambda}$
            after $k$ iterations, $\mat[n, k]{A}$ will denote $\mat[n]{A}$ after $k$
            iterations).
    \end{itemize}
\end{notation}

\begin{remark} \label{R:alg}
    \hfill
    \begin{enumerate}[label=\textup(\roman*\textup)]
        \item
            One possible stopping condition is that the error $\norm{\ten{X} - \ten[k]{X}}$
            be small, where $\ten[k]{X} \defeq \tuck[\vec[k]{\lambda}]{\mat[1, k]{A}, \dots,
            \mat[N, k]{A}}$. To avoid the cost of forming $\ten[k]{X}$, one can observe that
            \[
                \norm{\ten{X} - \ten[k]{X}}^2 
                = \norm{\ten{X}}^2 - 2 \inner{\ten{X}}{ \ten[k]{X}} + \norm{\ten[k]{X}}^2
            \]
            and make use of the fact that $\mat[k][(n)]{X} = \mat[n, k]{A} \mat[k]{\Lambda}
            (\mat[n, k+1]{K})^\tp$, where $\mat[k]{\Lambda} = \diag(\vec[k]{\lambda})$
            \cite{tensor-toolbox}.
            The norm of $\ten{X}$ is constant and need only be computed once, while the
            other two terms can be computed using matrices already formed in the course of
            each iteration:
            \begin{align*}
                \inner{\ten{X}}{\ten[k]{X}}
                &= \inner{\mat{X}_{(n)}}{\mat[k][(n)]{X}} \\
                &= \inner{\mat{X}_{(n)} \mat[n, k+1]{K}}{\mat[n, k]{A} \mat[k]{\Lambda}} \\
                &= \inner{\mat[n, k+1]{M}}{\mat[n, k]{A} \mat[k]{\Lambda}}
            \end{align*}
            and
            \begin{align*}
                \norm{\ten[k]{X}}^2
                &= \inner{\mat[k][(n)]{X}}{\mat[k][(n)]{X}} \\
                &= \inner{(\mat[n, k+1]{K})^\tp \mat[n, k+1]{K}}{(\mat[n, k]{A}
                \mat[k]{\Lambda})^\tp \mat[n, k]{A} \mat[k]{\Lambda}} \\
                &= \inner{\mat[n, k+1]{H}}{(\mat[k]{\Lambda})^\tp \mat[n, k+1]{G}
                \mat[k]{\Lambda}}.
            \end{align*}
            However, it should be noted that this method of computing the error is less
            accurate than a direct computation because it computes the square of the error.

        \item
            The normalization step, while not mathematically necessary for convergence in
            the sense below, is nevertheless advisable to prevent numerical
            underflow/overflow \cite{tensor-toolbox}.

        \item
            The updating of $\vec{\lambda}$ need only be performed for one (arbitrary) $n$
            (e.g., $n \in \argmin_{n \in [N]} I_n$) and consists of setting $\lambda_r
            \leftarrow \norm{\vec[n][r]{a}}$ for $r = 1, \dots, R$.
    \end{enumerate}
\end{remark}

Our first main result is that CP-AltLS converges locally \emph{polynomially with order
$N-1$} for $N$\textsuperscript{th}-order \emph{orthogonally} decomposable tensors.

\begin{theorem}[Local convergence of CP-AltLS for orthogonally decomposable tensors] 
    \label{T:odeco}
    Suppose that $\ten{X} \in \R^{I_1 \times \cdots \times I_N}$ \textup($N \geq 3$\textup)
    is orthogonally decomposable and write $\ten{X} = \tuck[\vec{\lambda}]{\mat[1]{A},
    \dots, \mat[N]{A}}$ for some $\vec{\lambda} \in \R^R$ and some $\mat[n]{A}
    \in \R^{I_n \times R}$ with orthonormal columns. In addition, let $\mat[n, k]{A}$ denote
    the approximation to $\mat[n]{A}$ produced by \Cref{CP-AltLS}
    applied to $\ten{X}$ after $k$ iterations, and define 
    \[
        \epsilon_k \defeq \max_{n \in [N],\, r \in [R]}
        \abs{\sin \angle(\vec[n, k][r]{a}, \vec[n][r]{a})}
    \]
    and $\kappa \defeq {\max_{r \in [R]} \abs{\lambda_r}} / {\min_{r \in [R]}
    \abs{\lambda_r}}$. Then there exists a constant $C > 0$ such that if $\epsilon_0 < C
    \kappa^{-\frac{1}{N-2}} R^{-\frac{1}{N-1}}$, then $\set{\epsilon_k}_{k=0}^\infty$
    converges polynomially to zero. More precisely,
    \begin{enumerate}[label=\textup(\alph*\textup)]
        \item
            $\epsilon_k \leq (c(\kappa, R) \cdot \epsilon_{k-1})^{N-1}$,
            where $c(\kappa, R) \approx \kappa^{\frac{1}{N-1}} R^{\frac{1}{2(N-1)}}$;
        \item
            $\epsilon_k \leq \rho^{(N-1)^k}$ for some $\rho < 1$.
    \end{enumerate}
\end{theorem}

\begin{remark}
    \hfill
    \begin{enumerate}[label=\textup(\roman*\textup)]
        \item
            \Cref{CP-AltLS} does not converge in general for $N = 2$ (i.e., when $\ten{X}$
            is a matrix). 
            For example, suppose that $\lambda_r \defeq 1$ for $r = 1, \dots,
            R$ and $\mat[n]{A} \defeq \mat{I} \in \R^{R \times R}$ for $n = 1, 2$ so
            that $\ten{X} = \mat{I} \in \R^{R \times R}$. In addition, let
            $\mat[n, k]{\check{A}}$ denote $\mat[n, k]{A}$ prior to the normalization step.
            Then, provided that $\mat[1, k-1]{A}$
            and $\mat[2, k-1]{A}$ are invertible, we have
            \begin{align*}
                \mat[1, k]{\check{A}} &= (\mat[2, k-1]{A})^{-\tp}, \\
                \mat[2, k]{\check{A}} &= (\mat[1, k-1]{A})^{-\tp}.
            \end{align*}
            Hence
            \begin{align*}
                \mat[1, k]{A} &= (\mat[2, k-1]{A})^{-\tp} \mat[1, k]{D}, \\
                \mat[2, k]{A} &= (\mat[1, k-1]{A})^{-\tp} \mat[2, k]{D},
            \end{align*}
            for some diagonal matrices $\mat[1, k]{D}, \mat[2, k]{D}$ with positive
            diagonal entries. Therefore
            \begin{align*}
                \mat[1, k+1]{\check{A}} &= (\mat[2, k]{A})^{-\tp} 
                = \mat[1, k-1]{A} (\mat[2, k]{D})^{-\tp}, \\
                \mat[2, k+1]{\check{A}} &= (\mat[1, k]{A})^{-\tp} 
                = \mat[2, k-1]{A} (\mat[1, k]{D})^{-\tp}.
            \end{align*}
            It follows that
            \begin{align*}
                \mat[1, k+1]{A} &= \mat[1, k-1]{A}, \\
                \mat[2, k+1]{A} &= \mat[1, k-1]{A},
            \end{align*}
            so the iteration is 2-periodic in this case.

        \item
            This theorem only shows that the \emph{angle between} $\operatorname{span}
            \set{\vec[n, k][r]{a}}$ and $\operatorname{span} \set{\vec[n][r]{a}}$
            tends to zero -- indeed, CP decompositions are not unique; for instance, the
            signs of the vectors $\vec[n][r]{a}$ can be altered by adjusting
            $\vec{\lambda}$ accordingly.
            For the same reason, the elements of $\vec[k]{\lambda}$ will converge to those
            of $\vec{\lambda}$ only up to signs; see \Cref{C:weights}.
    \end{enumerate}
\end{remark}

Our second main result is that CP-AltLS converges locally \emph{linearly} for
\emph{incoherently} decomposable tensors.

\begin{theorem}[Local convergence of CP-AltLS for incoherently decomposable tensors]
    \label{T:ideco}
    Suppose that $\ten{X} \in \R^{I_1 \times \cdots \times I_N}$ \textup($N \geq 3$\textup)
    is $\mu$-coherently decomposable and write $\ten{X} =
    \tuck[\vec{\lambda}]{\mat[1]{A}, \dots, \mat[N]{A}}$ for some $\vec{\lambda}
    \in \R^R$ and some $\mat[n]{A} \in \R^{I_n \times R}$ with normalized columns and
    coherence at most $\mu$. In addition, let $\mat[n, k]{A}$ denote the approximation to
    $\mat[n]{A}$ produced by \Cref{CP-AltLS} applied to $\ten{X}$ after
    $k$ iterations, and define 
    \[
        \epsilon_k \defeq \max_{n \in [N],\, r \in [R]}
        \abs{\sin \angle(\vec[n, k][r]{a}, \vec[n][r]{a})}
    \]
    and $\kappa \defeq {\max_{r \in [R]} \abs{\lambda_r}} / {\min_{r \in [R]}
    \abs{\lambda_r}}$.
    Then there exists a constant $C > 0$ such that if $\max \set{\epsilon_0, \mu} < C
    \kappa^{-\frac{1}{N-2}} R^{-\frac{2}{N-2}}$,
    then $\set{\epsilon_k}_{k=0}^\infty$ converges linearly to zero.
\end{theorem}

Moreover, in both the orthogonal and incoherent cases, the squared weights converge at the
same rate as the factor matrices.

\begin{corollary}[Convergence of weights in CP-AltLS]
    \label{C:weights}
    Under the convergence condition of \Cref{T:odeco} or that of \Cref{T:ideco}, we
    have $(\lambda_r^{(k)})^2 = (\lambda_r)^2 + \bigO{\epsilon_{k-1}}$ for all $r \in [R]$
    \textup(so the weights also converge, up to signs\textup).
\end{corollary}

\subsection{Local convergence of CP-AltLS for orthogonally decomposable tensors} \label{SS:odeco}

Let us first outline the proof of \Cref{T:odeco}.
We begin by noting that \emph{before normalization}, the approximate factor matrix $\mat[n,
k]{A}$ is equal to $\mat{X}_{(n)} \mat[n, k]{K} (\mat[n, k]{H})^\pinv$, where 
\begin{align*}
    \mat[n, k]{G} &= (\mat[n, k-1]{A})^\tp \mat[n, k-1]{A}, \\
    \mat[n, k]{K} &= \bigkr_{m \neq n} \mat[m, k-1]{A}, \\
    \mat[n, k]{H} &= \bighd_{m \neq n} \mat[m, k]{G}.
\end{align*}
Since $\mat{X}_{(n)} = \mat[n]{A} \mat{\Lambda} (\mat[n]{K})^\tp$, where
$\mat{\Lambda} \defeq \diag(\vec{\lambda})$ and $\mat[n]{K} \defeq \bigkr_{m \neq n}
\mat[m]{A}$, the approximate factor matrix prior to normalization can be written as
$\mat[n]{A} \mat[n, k]{V}$, where
\[
    \mat[n, k]{V} \defeq \mat{\Lambda} (\mat[n]{K})^\tp \mat[n, k]{K} (\mat[n, k]{H})^\pinv.
\]
% For orthogonally decomposable tensors, the columns of
% $\mat[n]{A}$ are orthonormal, which implies that \emph{after normalization}, we
% have
% \[
%     \mat[n, k]{A} = \mat[n]{A} \mat[n, k]{\hat{V}},
% \]
% where $\mat[n, k]{\hat{V}}$ denotes the column-normalized version of $\mat[n, k]{V}$.
The quantity $\epsilon_k$, then, is determined by the angles between the columns of
$\mat[n]{A} \mat[n, k]{V}$ and those of $\mat[n]{A}$ because these angles are
preserved when the columns of the former are normalized.

To show that $\epsilon_k$ tends to zero, we successively estimate these matrix
products. For the purpose of our analysis, we may assume that the angles between the columns
are acute (by adjusting the signs of the elements of $\vec{\lambda}$), so that smallness
of $\epsilon_{k-1}$ in fact means that $\mat[n, k-1]{A} \approx \mat[n]{A}$. From
this and the fact that $(\mat[n]{A})^\tp \mat[n]{A} = \mat{I}$, we deduce that
\begin{gather*}
    (\mat[n]{A})^\tp \mat[n, k-1]{A} \approx \mat{I}, \\
    \mat[n, k]{G} = (\mat[n, k-1]{A})^\tp \mat[n, k-1]{A} \approx \mat{I},
\end{gather*}
in a way that is quantified by \Cref{L:innerprod-o,L:innerprod-i}. 
It follows from the nature of the Khatri--Rao and Hadamard products that
\begin{gather*}
    (\mat[n]{K})^\tp \mat[n, k]{K} = \bighd_{m \neq n} (\mat[m]{A})^\tp \mat[m, k-1]{A} 
    \approx \mat{I}, \\
    \mat[n, k]{H} = \bighd_{m \neq n} \mat[m, k]{G}
    \approx \mat{I}.
\end{gather*}
Consequently,
\begin{gather*}
    (\mat[n, k]{H})^\pinv = (\mat[n, k]{H})^{-1} \approx \mat{I},
\end{gather*}
as concretized by \Cref{L:inverse} and \Cref{C:inverse}. We combine these estimates using
\Cref{L:product-o} to obtain
\begin{gather*}
    \mat{\Lambda}^{-1} \mat[n, k]{V} = 
    (\mat[n]{K})^\tp \mat[n, k]{K} (\mat[n, k]{H})^\pinv \approx \mat{I}.
\end{gather*}
Finally, \Cref{L:normalization-o} implies that $\epsilon_k$, in turn, is small, essentially
because $\mat[n, k]{V} \approx \mat{\Lambda}$ so that the columns of $\mat[n]{A}
\mat[n, k]{V}$ are approximately parallel to those of $\mat[n]{A}$. It is then a
matter of showing that these estimates hold, and that $\epsilon_k$ decreases, when
$\epsilon_0$ is sufficiently small.

\begin{proof}[Proof of \Cref{T:odeco}]
    Without loss of generality, we may assume that $\inner{\vec[n,
    k][r]{a}}{\vec[n][r]{a}} \allowbreak\geq 0$ for all $n$ and $r$ by absorbing the signs of the
    inner products into $\vec{\lambda}$, because this preserves the quantities $\abs{\sin
    \angle(\vec[n, k][r]{a}, \vec[n][r]{a})}$.

    By \Cref{L:innerprod-o,L:innerprod-i} applied to $\mat[n, k-1]{A}$ and $\mat[n]{A}$, we have
    \begin{align*}
        \norm{(\mat[n]{A})^\tp \mat[n, k-1]{A} - \mat{I}}_\mathrm{max} &\leq \sqrt{2} \epsilon_{k-1}, \\
        \norm{\mat[n, k]{G} - \mat{I}}_\mathrm{max} &\leq 2\sqrt{2}\epsilon_{k-1}, \\
        \norm{\D'((\mat[n]{A})^\tp \mat[n, k-1]{A})}_\mathrm{1, 2} &\leq \epsilon_{k-1}.
    \end{align*}
    Since $(\mat[n]{K})^\tp \mat[n, k]{K} = \bighd_{m \neq n} (\mat[m]{A})^\tp
    \mat[m, k-1]{A}$, it follows from the first and third inequalities that
    \begin{align*}
        \norm{\D((\mat[n]{K})^\tp \mat[n, k]{K}) - \mat{I}}_\mathrm{max} 
        &\leq 1 - (1 - \sqrt{2}\epsilon_{k-1})^{N-1}, \\
        \norm{\D'((\mat[n]{K})^\tp \mat[n, k]{K})}_{1, 2} 
        &\leq \epsilon_{k-1}^{N-1},
    \end{align*}
    assuming that
    \[
        \sqrt{2}\epsilon_{k-1} \leq 1.
    \]
    On the other hand, since $\mat[n, k]{H} = \bighd_{m \neq n} \mat[m, k]{G}$, it follows
    from the second inequality that $\norm{\D'(\mat[n, k]{H})}_\mathrm{max} \leq
    (2\sqrt{2}\epsilon_{k-1})^{N-1}$. Hence, by \Cref{C:inverse}, $\mat[n, k]{H}$ is
    invertible and
    \[
        \norm{(\mat[n, k]{H})^{-1} - \mat{I}}_\mathrm{max} \leq 
        \frac{(2\sqrt{2} \epsilon_{k-1})^{N-1}}{1 - (R-1)(2\sqrt{2} \epsilon_{k-1})^{N-1}}
    \]
    provided that
    \[
        (R-1)(2\sqrt{2} \epsilon_{k-1})^{N-1} < 1.
    \]

    Now suppose that $R(2\sqrt{2} \epsilon_{k-1})^{N-1} \leq C$ for some constant $0 < C \ll 1$ to
    be determined later (which implies in particular that $2\sqrt{2} \epsilon_{k-1} \leq 1$). Then
    \begin{align*}
        \norm{\D((\mat[n]{K})^\tp \mat[n, k]{K}) - \mat{I}}_\mathrm{max} 
        &\leq 1 - \left(\frac{1}{2}\right)^{N-1} \eqdef \epsilon_\mathrm{K}, \\
        \norm{\D'((\mat[n]{K})^\tp \mat[n, k]{K})}_{1, 2} 
        &\leq \epsilon_{k-1}^{N-1} \eqdef \epsilon_\mathrm{K}', \\
        \norm{(\mat[n, k]{H})^{-1} - \mat{I}}_\mathrm{max} 
        &\leq 
        \frac{(2\sqrt{2}\epsilon_{k-1})^{N-1}}{1 - C} \eqdef \epsilon_\mathrm{H} \leq
        \frac{C}{1-C}.
    \end{align*}
    Combining these estimates for $(\mat[n]{K})^\tp \mat[n, k]{K}$ and 
    $(\mat[n, k]{H})^{-1}$ using \Cref{L:product-o}, we obtain
    \begin{align*}
        \norm{\D(\mat{\Lambda}^{-1} \mat[n, k]{V}) - \mat{I}}_\mathrm{max}
        &\leq
        1 - (1 - \epsilon_\mathrm{K})(1 - \epsilon_\mathrm{H}) + 
        (R-1) \epsilon_\mathrm{K}' \epsilon_\mathrm{H}^{} \\
        &\leq
        1 - \left(\frac{1}{2}\right)^{N-1}\left(1 - \frac{C}{1-C}\right) + 
        \frac{C}{(2\sqrt{2})^{N-1}} \cdot \frac{C}{1-C} \\
        &\leq
        1 - \left(1 - \frac{C}{1-C} - \frac{C^2}{1-C}\right)\left(\frac{1}{2}\right)^{N-1} \\
    \end{align*}
    and
    \begin{align*}
        \norm{\D'(\mat{\Lambda}^{-1} \mat[n, k]{V})}_\mathrm{1, 2}
        &\leq
        (R-1)^{\frac{1}{2}} \epsilon_\mathrm{H}^{} + 
        \epsilon_\mathrm{K}' (1 + \epsilon_\mathrm{H}^{}) +
        (R-1) \epsilon_\mathrm{K}' \epsilon_\mathrm{H}^{} \\
        &=
        (R-1)^{\frac{1}{2}} \epsilon_\mathrm{H}^{} + 
        \epsilon_\mathrm{K}' (1 + R\epsilon_\mathrm{H}^{}) \\
        &\leq
        R^{\frac{1}{2}} \cdot \frac{(2\sqrt{2}\epsilon_{k-1})^{N-1}}{1-C} + 
        (2\sqrt{2}\epsilon_{k-1})^{N-1} \left(1 + \frac{C}{1-C}\right) \\
        &\leq
        \left(\frac{1}{1-C} + 1 + \frac{C}{1-C}\right)
        R^{\frac{1}{2}} (2\sqrt{2}\epsilon_{k-1})^{N-1}.
    \end{align*}
    Finally, applying \Cref{L:normalization-o} to $\mat[n]{A} \mat[n, k]{V}$, we arrive at
    \[
        \epsilon_k
        % =
        % \norm{\D'((\mat[n]{A})^\tp \mat[n, k]{A})}_{1, 2}
        \leq 
        \frac{2}{1 - 2C - C^2} \cdot
        \kappa R^{\frac{1}{2}} (4\sqrt{2}\epsilon_{k-1})^{N-1}.
    \]
    One can verify that the choice $C = \frac{1}{3}$ is sufficiently small for the lemmas to
    be applicable, which yields
    \begin{equation} \label{polyconv}
        \epsilon_k
        \leq 
        9
        \kappa R^{\frac{1}{2}} (4\sqrt{2}\epsilon_{k-1})^{N-1}
        \quad
        \text{if $R(2\sqrt{2}\epsilon_{k-1})^{N-1} \leq \frac{1}{3}$.}
    \end{equation}

    Indeed, if $\epsilon_0 < C \kappa^{-\frac{1}{N-2}} R^{-\frac{1}{N-1}}$ for some constant
    $C > 0$ (not necessarily equal to the one above) with $9(4\sqrt{2})^2 \cdot C \leq
    1$, then $R(2\sqrt{2}\epsilon_0)^{N-1} < (2\sqrt{2}C \kappa^{-\frac{1}{N-2}})^{N-1} < (3C)^{N-1} \leq
    \frac{1}{3}$.  Moreover, $9 \kappa R^{\frac{1}{2}} (4\sqrt{2} \epsilon_{k-1})^{N-1}
    \leq \epsilon_{k-1}$ if and only if $\epsilon_{k-1} \leq 9^{-\frac{1}{N-2}} \cdot \allowbreak
    (4\sqrt{2})^{-\frac{N-1}{N-2}} \kappa^{-\frac{1}{N-2}} R^{-\frac{1}{2(N-2)}}$, which
    holds for $k = 1$ since $N \geq 3$. Thus, by bootstrapping, inequality
    \labelcref{polyconv} holds for all $k$, which proves part (a). For part (b), we
    iterate this inequality $k$ times to conclude that
    \[
        \epsilon_k 
        \leq 
        \left(9\kappa R^{\frac{1}{2}} (4\sqrt{2})^{N-1}\right)^{\frac{(N-1)^k
        - 1}{N-2}} \epsilon_0^{(N-1)^k} 
        \leq 
        \left(9\kappa^{\frac{1}{N-2}} R^{\frac{1}{2(N-2)}} (4\sqrt{2})^2
        \epsilon_0\right)^{(N-1)^k}. \qedhere
    \]
\end{proof}

% \begin{lemma}
%     Let $\mat{A} \in \R^{n \times n}$. If $\norm{\D(\mat{\Lambda}^{-1} \mat{A}) -
%     \mat{I}}_\mathrm{max} \leq \epsilon < 1$ and $\norm{\D'(\mat{\Lambda}^{-1}
%     \mat{A})}_\mathrm{1, 2} \leq \epsilon'$ for some diagonal matrix $\mat{\Lambda} =
%     \diag(\lambda_1, \dots, \lambda_n)$, then the column-normalized version
%     $\mat{\hat{A}}$ of $\mat{A}$ satisfies
%     \[
%         \norm{\D'(\mat{\hat{A}})}_\mathrm{1, 2} 
%         \leq 
%         \frac{\kappa \epsilon'}{1 - \epsilon},
%     \]
%     where $\kappa \defeq \max_{j \in [n]} \abs{\lambda_j} / \min_{j \in [n]}
%     \abs{\lambda_j}$.
% \end{lemma}
% 
% \begin{proof}
%     By the first part of the hypothesis,
%     $\abs{a_{jj}/\lambda_j - 1} \leq \epsilon < 1$ for all $j \in [n]$, so $\abs{a_{jj}}
%     \geq \abs{\lambda_j} - \abs{a_{jj} - \lambda_j} \geq (1 - \epsilon) \abs{\lambda_j} \geq
%     (1 - \epsilon) \min_j \abs{\lambda_j}$. By the second part of the hypothesis, $(\sum_{i
%     \neq j} \abs{a_{ij}/\lambda_i}^2)^{\frac{1}{2}} \leq \epsilon'$ for all $j \in [n]$, so
%     $(\sum_{i \neq j} \abs{a_{ij}}^2)^{\frac{1}{2}} \leq \epsilon' \max_j \abs{\lambda_j}$.
%     The conclusion then follows from these inequalities along with the fact that
%     $\norm{\D'(\mat{\hat{A}})}_{1, 2} = \max_{j \in [n]} (\sum_{i \neq j}
%     \abs{a_{ij}}^2/\norm{\vec{a}_j}^2)^{\frac{1}{2}} \leq \max_{j \in [n]} (\sum_{i \neq j}
%     \abs{a_{ij}}^2/\abs{a_{jj}}^2)^{\frac{1}{2}}$.
% \end{proof}

\subsection{Local convergence of CP-AltLS for incoherently decomposable tensors} \label{SS:ideco}

Our approach in the incoherent case is the same as that in the orthogonal case, with the
requisite modifications to account for the coherence of the factor matrices. In particular, it is
useful to define
\begin{align*}
    \mat[n]{G} &\defeq (\mat[n]{A})^\tp \mat[n]{A}, \\
    \mat[n]{H} &\defeq \bighd_{m \neq n} \mat[m]{G},
\end{align*}
by analogy with $\mat[n, k]{G}$ and $\mat[n, k]{H}$, as these matrices are no longer
necessarily the identity matrix.
The estimate $\mat[n, k-1]{A} \approx \mat[n]{A}$ derived from the smallness of
$\epsilon_{k-1}$ then leads to
\begin{gather*}
    (\mat[n]{A})^\tp \mat[n, k-1]{A} \approx \mat[n]{G}, \\
    \mat[n, k]{G} = (\mat[n, k-1]{A})^\tp \mat[n, k-1]{A} \approx \mat[n]{G}.
\end{gather*}
Forming the Khatri--Rao and Hadamard products yields
\begin{gather*}
    (\mat[n]{K})^\tp \mat[n, k]{K} 
    = \bighd_{m \neq n} (\mat[m]{A})^\tp \mat[m, k-1]{A}
    \approx \bighd_{m \neq n} \mat[m]{G} 
    = (\mat[n]{K})^\tp \mat[n]{K}, \\
    \mat[n, k]{H} 
    = \bighd_{m \neq n} \mat[m, k]{G}
    \approx \bighd_{m \neq n} \mat[m]{G} = \mat[n]{H}.
\end{gather*}
Consequently,
\begin{gather*}
    (\mat[n, k]{H})^\pinv = (\mat[n, k]{H})^{-1} \approx 
    (\mat[n]{H})^{-1}.
\end{gather*}
We combine these estimates using \Cref{L:product-i} to obtain
\begin{gather*}
    \mat{\Lambda}^{-1} \mat[n, k]{V} 
    = (\mat[n]{K})^\tp \mat[n, k]{K} (\mat[n, k]{H})^\pinv 
    \approx (\mat[n]{K})^\tp \mat[n]{K} (\mat[n]{H})^{-1} 
    = \mat{I}.
\end{gather*}
Finally, we conclude via \Cref{L:normalization-i} that $\epsilon_k$ is small.

\begin{proof}[Proof of \Cref{T:ideco}]
    As in the proof of \Cref{T:odeco}, we assume without loss of generality that
    $\inner{\vec[n, k][r]{a}}{\vec[n][r]{a}} \geq 0$ for all $n$ and $r$.

    By \Cref{L:innerprod-i} applied to $\mat[n, k-1]{A}$ and $\mat[n]{A}$, we have
    \begin{align*}
        \norm{(\mat[n]{A})^\tp \mat[n, k-1]{A} - \mat[n]{G}}_\mathrm{max} &\leq
        \sqrt{2} \epsilon_{k-1}, \\
        \norm{\mat[n, k]{G} - \mat[n]{G}}_\mathrm{max} &\leq 2\sqrt{2} \epsilon_{k-1}.
    \end{align*}
    Since $(\mat[n]{K})^\tp \mat[n, k]{K} = \bighd_{m \neq n} (\mat[m]{A})^\tp
    \mat[m, k-1]{A}$ and $(\mat[n]{K})^\tp \mat[n]{K} = \bighd_{m \neq n}
    \mat[m]{G}$, it follows from the first inequality and the triangle inequality that
    \begin{align*}
        % \norm{(\mat[n]{K})^\tp \mat[n, k]{K} - (\mat[n]{K})^\tp \mat[n]{K}}_\mathrm{max} 
        % &\leq \sqrt{2} (N-1) \epsilon_{k-1}, \\
        \norm{\D((\mat[n]{K})^\tp \mat[n, k]{K} - (\mat[n]{K})^\tp
        \mat[n]{K})}_\mathrm{max} 
          &\leq 1 - (1 - \sqrt{2} \epsilon_{k-1})^{N-1} \eqdef \epsilon_\mathrm{K}, \\
        \norm{\D'((\mat[n]{K})^\tp \mat[n, k]{K} - (\mat[n]{K})^\tp
        \mat[n]{K})}_\mathrm{max}
          &\leq \alpha \epsilon_{k-1} \eqdef \epsilon_\mathrm{K}',
    \end{align*}
    where
    \[
        \alpha \defeq \sqrt{2}(N-1)(\sqrt{2}\epsilon_{k-1} + \mu)^{N-2}.
    \]
    On the other hand, since $\mat[n, k]{H} = \bighd_{m \neq n} \mat[m, k]{G}$ and
    $\mat[n]{H} = \bighd_{m \neq n} \mat[m]{G}$, it follows from the second
    inequality and the triangle inequality that $\norm{\D'(\mat[n, k]{H} -
    \mat[n]{H})}_\mathrm{max} \leq 2\alpha \epsilon_{k-1}.$ Hence, by
    \Cref{L:inverse}, $\mat[n, k]{H}$ and $\mat[n]{H}$ are invertible and
    \begin{align*}
        \norm{(\mat[n, k]{H})^{-1} - (\mat[n]{H})^{-1}}_\mathrm{max} 
        &\leq 
        \frac{2\alpha\epsilon_{k-1}}{[1 - (R-1)(2\alpha\epsilon_{k-1} +
        \mu^{N-1})]^2} \eqdef \epsilon_\mathrm{H}, \\
        \norm{(\mat[n]{H})^{-1}}_\mathrm{max} 
        &\leq 
        \frac{1}{1 - (R-1)\mu^{N-1}}, \\
        \norm{\D'((\mat[n]{H})^{-1})}_\mathrm{max} 
        &\leq 
        \frac{(R-1)\mu^{N-1}}{1 - (R-1)\mu^{N-1}},
    \end{align*}
    provided that
    \[
        (R-1)(2\alpha\epsilon_{k-1} + \mu^{N-1}) < 1.
    \]

    Now suppose that $\max \set{\epsilon_{k-1}, \mu} \leq C \kappa^{-\frac{1}{N-2}}
    R^{-\frac{2}{N-2}}$ 
    for some constant $0 < C \ll 1$ to be determined later, and for the sake of definiteness,
    assume that
    \begin{equation} \label{cond1}
        R \cdot 2\alpha \epsilon_{k-1} \leq \frac{1}{4},\quad
        R \cdot \mu^{N-1} \leq \frac{1}{4}.
    \end{equation}
    Using \Cref{L:product-i} to bound the norms of
    \[
        (\mat[n]{K})^\tp \mat[n, k]{K} (\mat[n, k]{H})^{-1} - (\mat[n]{K})^\tp
        \mat[n]{K} \allowbreak (\mat[n]{H})^{-1} =
        \mat{\Lambda}^{-1} \mat[n, k]{V} - \mat{I},
    \]
    we obtain
    \begin{align*}
        \norm{\D(\mat{\Lambda}^{-1} \mat[n, k]{V}) - \mat{I}}_\mathrm{max}
        &\leq
        (R-1)(\alpha \epsilon_{k-1} + \mu^{N-1}) \cdot \epsilon_\mathrm{H}^{} +
        \epsilon_\mathrm{H}^{} + {} \\
        &\hphantom{{}\leq{}}
        (R-1) \frac{(R-1) \mu^{N-1}}{1 - (R-1) \mu^{N-1}} \cdot \epsilon_\mathrm{K}' +
        \frac{1}{1 - (R-1) \mu^{N-1}} \cdot \epsilon_\mathrm{K}^{} \\
        &\leq
        \frac{\epsilon_\mathrm{K}}{1 - R \mu^{N-1}} +
        \left(\frac{1}{3} R + 11\right) \alpha \epsilon_{k-1} \\
        &\leq
        \frac{\epsilon_\mathrm{K}}{1 - R \mu^{N-1}} +
        \frac{34}{3} R \alpha \epsilon_{k-1}
        \eqdef \epsilon \stepcounter{equation}\tag{\theequation}\label{diag-lambda}
    \end{align*}
    and
    \begin{align*}
        \norm{\D'(\mat{\Lambda}^{-1} \mat[n, k]{V})}_\mathrm{max}
        &\leq
        (R-1) \cdot \epsilon_\mathrm{H} + 
        (\alpha \epsilon_{k-1} + \mu^{N-1}) \cdot \epsilon_\mathrm{H}^{} + {} \\
        &\hphantom{{}\leq{}}
        (R-1) \frac{1}{1 - (R-1) \mu^{N-1}} \cdot \epsilon_\mathrm{K}' + 
        \frac{(R-1) \mu^{N-1}}{1 - (R-1) \mu^{N-1}} \cdot \epsilon_\mathrm{K}^{} \\
        &\leq
        \frac{4}{3} R \mu^{N-1} \epsilon_\mathrm{K} +
        \left(\frac{28}{3}R + 3\right) \alpha \epsilon_{k-1} \\
        &\leq
        \frac{4}{3} R \mu^{N-1} \sqrt{2}(N-1) \epsilon_{k-1} +
        \left(\frac{28}{3}R + 3\right) \alpha \epsilon_{k-1} \\
        &\leq
        \frac{41}{3} R \alpha \epsilon_{k-1}
        \eqdef \epsilon'. \stepcounter{equation}\tag{\theequation}\label{offdiag-lambda}
    \end{align*}
    Then, applying \Cref{L:normalization-i}, we arrive at
    \[
        \epsilon_k
        \leq 
        \left(\frac{2}{(1-\epsilon)^2 - 4(R-1)\kappa\epsilon' - [(R-1)\kappa\epsilon']^2}\right)^\frac{1}{2}
        (R-1)\kappa \epsilon',
    \]
    provided that
    \[
        \epsilon < 1, \quad
        4(R-1)\kappa\epsilon' + [(R-1)\kappa\epsilon']^2 < (1 - \epsilon)^2.
    \]
    Thus, if 
    \[
        \beta \defeq \frac{41}{3} \kappa R^2 \alpha
    \]
    so that $(R-1) \kappa \epsilon' \leq \beta \epsilon_{k-1} \leq \beta$, and if we also assume that
    \begin{equation} \label{cond2}
        (13 \beta)^{\frac{1}{2}} + \epsilon \leq 1,
    \end{equation}
    then we have linear convergence:
    \[
        \epsilon_k \leq 
        \left(\frac{2}{(1-\epsilon)^2 - 5\beta}\right)^\frac{1}{2}
        \beta \epsilon_{k-1} \leq
        \left(\frac{2}{8 \beta}\right)^\frac{1}{2}
        \beta^{\frac{1}{2}} \epsilon_{k-1} =
        \frac{1}{2} \epsilon_{k-1}.
    \]

    It remains to show that conditions \labelcref{cond1,cond2} are satisfied for a
    sufficiently small choice of $C$. Since $\alpha \leq ((4 + 2\sqrt{2}) \max
    \set{\epsilon_{k-1}, \mu})^{N-2}$,
    condition \labelcref{cond1} is satisfied when $C \leq [8(4 + 2\sqrt{2})]^{-1} \approx
    1.8 \times 10^{-2}$.
    As for condition \labelcref{cond2}, we note that
    \begin{align*}
        (13\beta)^\frac{1}{2} + \epsilon
        &\leq
        \frac{(13\beta)^\frac{1}{2} + 1 - (1 - \sqrt{2} C)^{N-1} + 
        \frac{34}{3} R\alpha\epsilon_{k-1}}{1 - R\mu^{N-1}} \\
        &\leq
        \frac{(13\beta)^\frac{1}{2} + 1 - i(1 - \sqrt{2} C)^{N-1} + 
        \beta^\frac{1}{2}}{1 - \beta^\frac{1}{2}},
    \end{align*}
    provided that $\beta < 1$, so this condition will be satisfied if
    \[
        (\sqrt{13} + 2)^\frac{1}{N-1} \beta^\frac{1}{2(N-1)} + \sqrt{2} C
        \leq 1.
    \]
    Finally,
    \[
        (\sqrt{13} + 2)^\frac{1}{N-1} \beta^\frac{1}{2(N-1)} + \sqrt{2} C
        \leq
        (\sqrt{13} + 2)^\frac{1}{2} \left(\frac{41}{3} (4 + 2\sqrt{2}) C\right)^\frac{1}{4}
        + \sqrt{2} C,
    \]
    where the right-hand side is less than $1$ for $C \leq
    [(\sqrt{13} + 2)^\frac{1}{2}  (\frac{41}{3} (4 + 2\sqrt{2}))^\frac{1}{4} +
    \sqrt{2}]^{-4} \approx 1.7 \times 10^{-4}$.
\end{proof}

\begin{proof}[Proof of \Cref{C:weights}]
    We consider the incoherent case only; the interested reader may verify that the same
    argument applies in the orthogonal case. 

    By definition,
    \[
        (\lambda_r^{(k)})^2 
        = \left((\mat[n]{A} \mat[n, k]{V})^\tp (\mat[n]{A} \mat[n, k]{V})\right)_{rr}
        = \left((\mat[n, k]{V})^\tp \mat[n]{G} (\mat[n, k]{V})\right)_{rr}
    \]
    for some $n$ (see \Cref{R:alg}), 
    and from estimates \labelcref{diag-lambda,offdiag-lambda} in
    the proof of \Cref{T:ideco}, we know that $v_{ii}^{(n, k)} = \lambda_i +
    \bigO{\epsilon_{k-1}}$ and $v_{ij}^{(n, k)} = \bigO{\epsilon_{k-1}}$ for $j \neq i$.
    Hence (omitting superscripts for readability)
    \begin{align*}
        (\lambda_r^{(k)})^2
        = \sum_{i,\,j} v_{ir} g_{ij} v_{jr}
        &= v_{rr} g_{rr} v_{rr} +
        \sum_{\substack{i,\,j \\ i = r,\,j \neq r}} v_{ir} g_{ij} v_{jr} +
        \sum_{\substack{i,\,j \\ i \neq r}} v_{ir} g_{ij} v_{jr} \\
        &= (\lambda_r)^2 + \bigO{\epsilon_{k-1}}.
        \qedhere
    \end{align*}
\end{proof}

\subsection{Local convergence of CP-SAltLS} \label{SS:convergence-saltls}

The arguments in the preceding subsections in fact show that \Cref{T:odeco,T:ideco} hold for
CP-SAltLS (\Cref{CP-SAltLS}) as well. Indeed, in the serial algorithm, the approximate
factor matrices prior to normalization are still equal to
\[
    \mat[n]{A} \mat[n, k]{V} 
    = \mat[n]{A} \mat{\Lambda} (\mat[n]{K})^\tp \mat[n, k]{K} (\mat[n, k]{H})^\pinv.
\]
However,
\begin{align*}
    \mat[n, k]{K} &= \left(\bigkr_{m > n} \mat[m, k-1]{A}\right) \kr \left(\bigkr_{m < n}
    \mat[m, k]{A}\right), \\
    \mat[n, k]{H} &= \left(\bighd_{m > n} \mat[m, k]{G}\right) \hd \left(\bighd_{m < n}
    \mat[m, k-1]{G}\right), \\
    \mat[n, k]{G} &= (\mat[n, k]{A})^\tp \mat[n, k]{A}.
\end{align*}
Nevertheless, in the orthogonal case, we have
\begin{align*}
    \norm{(\mat[n]{A})^\tp \mat[n, k]{A} - \mat{I}}_\mathrm{max} &\leq \sqrt{2} \epsilon_{k}, &
    \norm{(\mat[n]{A})^\tp \mat[n, k-1]{A} - \mat{I}}_\mathrm{max} &\leq \sqrt{2} \epsilon_{k-1}, \\
    \norm{\mat[n, k]{G} - \mat{I}}_\mathrm{max} &\leq 2\sqrt{2}\epsilon_{k}, &
    \norm{\mat[n, k-1]{G} - \mat{I}}_\mathrm{max} &\leq 2\sqrt{2}\epsilon_{k-1}, \\
    \norm{\D'((\mat[n]{A})^\tp \mat[n, k]{A})}_\mathrm{1, 2} &\leq \epsilon_{k}, &
    \norm{\D'((\mat[n]{A})^\tp \mat[n, k-1]{A})}_\mathrm{1, 2} &\leq \epsilon_{k-1},
\end{align*}
so the same proof applies if $\epsilon_k \leq \epsilon_{k-1}$, which is guaranteed by
the hypothesis on the initial approximate factor matrices.
In the incoherent case, we have
\begin{align*}
    \norm{(\mat[n]{A})^\tp \mat[n, k]{A} - \mat[n]{G}}_\mathrm{max} &\leq \sqrt{2} \epsilon_{k}, &
    \norm{(\mat[n]{A})^\tp \mat[n, k-1]{A} - \mat[n]{G}}_\mathrm{max} &\leq \sqrt{2} \epsilon_{k-1}, \\
    \norm{\mat[n, k]{G} - \mat[n]{G}}_\mathrm{max} &\leq 2\sqrt{2} \epsilon_{k}, &
    \norm{\mat[n, k-1]{G} - \mat[n]{G}}_\mathrm{max} &\leq 2\sqrt{2} \epsilon_{k-1},
\end{align*}
and can thus argue similarly.

\subsection{Global convergence of CP-AltLS} \label{SS:convergence-global}

Although the convergence of CP-AltLS depends on the initial guesses for the factor matrices, we can
mitigate this in principle by restarting the algorithm with random initial guesses if it
fails to converge, thereby ensuring that the algorithm converges with high probability if
sufficiently many initializations are performed. We exemplify this principle below for
small tensors, for which the exact probabilities involved are straightforward to calculate.

\begin{proposition}[Global convergence of CP-AltLS for small cubical orthogonally decomposable tensors]
    Let $\ten{X} = \tuck[\vec{\lambda}]{\mat[1]{A}, \dots, \mat[N]{A}}$
    \textup($N \geq 3$\textup) for some $\vec{\lambda} \in \R^R$ and some $\mat[n]{A}
    \in \R^{2 \times R}$ with orthonormal columns.
    Suppose that \Cref{CP-AltLS} is applied to $\ten{X}$ with $\vec[n,
    0][1]{a}$ drawn randomly from the uniform distribution on $S^1$ \textup(for each $n \in
    [N]$\textup) and, if $R = 2$, $\vec[n, 0][2]{a}$ chosen orthonormal to $\vec[n,
    0][1]{a}$.
    For any $\delta \in (0, 1)$, if more than 
    \[
        \frac{\log(\delta)}{\log\left(1 - (\frac{2C}{\pi})^N \kappa^{-3} R^{-\frac{3}{2}}\right)}
    \]
    initializations are performed,
    where $C$ and $\kappa$ are as in \Cref{T:odeco},
    then with probability at least $1-\delta$, 
    the algorithm will converge for one of these initializations.
\end{proposition}

\begin{proof} 
    Let $\epsilon_0$ be as in \Cref{T:odeco} and $\epsilon \defeq C \kappa^{-\frac{1}{N-2}}
    R^{-\frac{1}{N-1}}$ so that the algorithm converges if $\epsilon_0 < \epsilon$. If the
    angle between the span of $\vec[n, 0][1]{a}$ and that of $\vec[n][1]{a}$ is less
    than $\epsilon$, then $\max_{r \in [R]} \abs{\sin \angle(\vec[n, 0][1]{a},
    \vec[n][1]{a})} \allowbreak < \epsilon$, and the probability of this occurring is
    $\frac{4\epsilon}{2\pi}$. Hence the probability of a convergent initialization is at
    least $(\frac{2\epsilon}{\pi})^N \geq (\frac{2C}{\pi})^N \kappa^{-3} R^{-\frac{3}{2}}$,
    from which the result follows.
\end{proof}

\section{Experiments} \label{S:experiments}

\begin{note}
    In the following experiments, a \emph{random vector\textup/matrix} will refer to a
    vector/matrix with entries drawn independently from the standard normal distribution,
    unless otherwise specified.
\end{note}

\subsection{Orthogonally decomposable tensors} \label{SS:exp-odeco}

To verify \Cref{T:odeco} empirically, random weights $\vec{\lambda} \in \R^{10}$ and factor matrices
$\mat[n]{A} \in \R^{20 \times 10}$ were generated, with the factor matrices taken from the QR
factorizations of random $20 \times 20$ matrices. 
Each factor matrix was perturbed by a (distinct) random matrix scaled by $10^{-2}$ and was then
normalized columnwise to form the initial approximate factor matrices $\mat[n, 0]{A} \in \R^{20 \times 10}$.
\Cref{F:odeco} shows the result of applying \Cref{CP-AltLS} when $N = 3$ and when $N = 4$,
in which case quadratic convergence and cubic convergence, respectively, are apparent.

\begin{figure}[htbp]
    \centering

    \begin{subfigure}{0.5\textwidth}
        \centering
        \includegraphics[width=\textwidth]{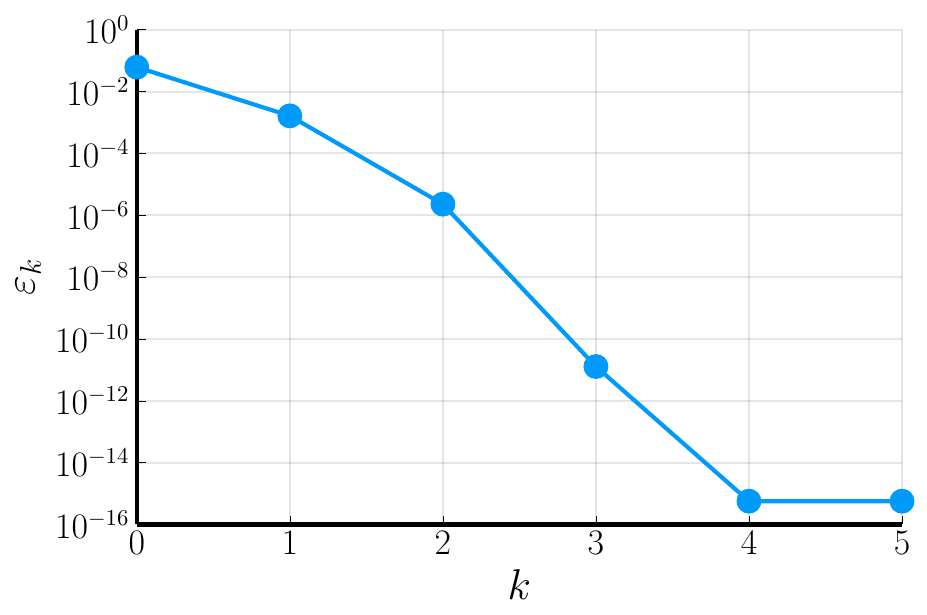}
        \caption{$N = 3$}
        \label{F:odeco-3}
    \end{subfigure}%
    \begin{subfigure}{0.5\textwidth}
        \centering
        \includegraphics[width=\textwidth]{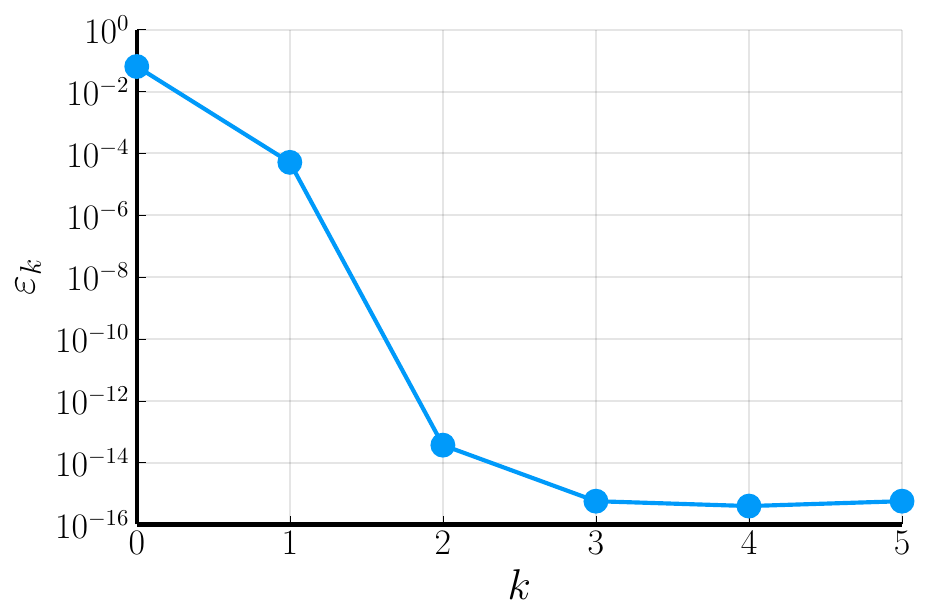}
        \caption{$N = 4$}
        % \label{}
    \end{subfigure}

    \caption{Polynomial convergence of \Cref{CP-AltLS} for $N$\textsuperscript{th}-order
    orthogonally decomposable tensors.}
    \label{F:odeco}
\end{figure}

\subsection{Incoherently decomposable tensors} \label{SS:exp-ideco}

To verify \Cref{T:ideco} empirically, random weights and factor matrices were generated as in
\Cref{SS:exp-odeco}, except that each exact factor matrix $\mat[n]{A}$ was also perturbed by a
(distinct) random matrix scaled by $10^{-2}$ to introduce incoherence.
\Cref{F:ideco} shows the result of applying \Cref{CP-AltLS} when $N = 3$ and when $N = 4$;
in both cases, linear convergence is apparent.

\begin{figure}[htbp]
    \centering

    \begin{subfigure}{0.5\textwidth}
        \centering
        \includegraphics[width=\textwidth]{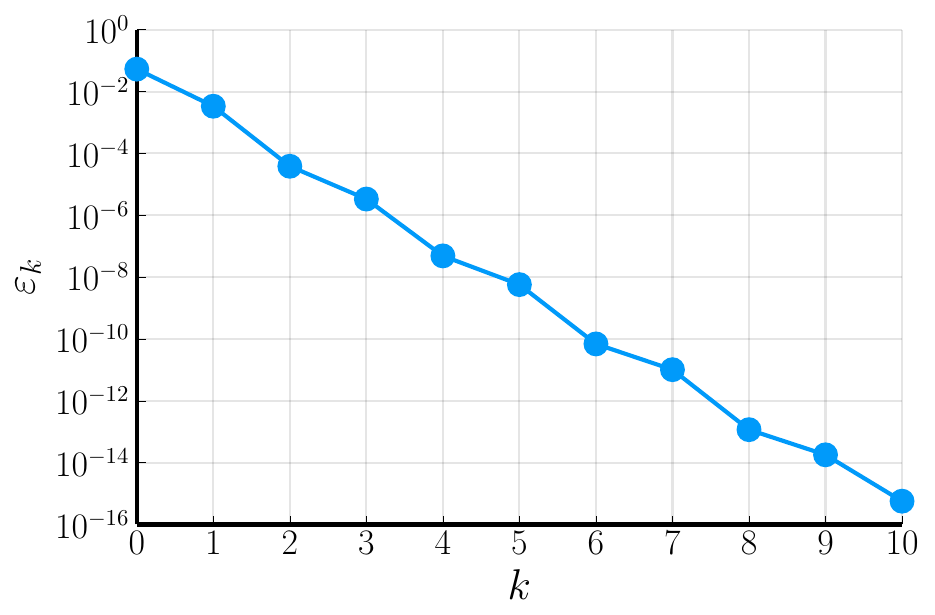}
        \caption{$N = 3$}
        \label{F:ideco-3}
    \end{subfigure}%
    \begin{subfigure}{0.5\textwidth}
        \centering
        \includegraphics[width=\textwidth]{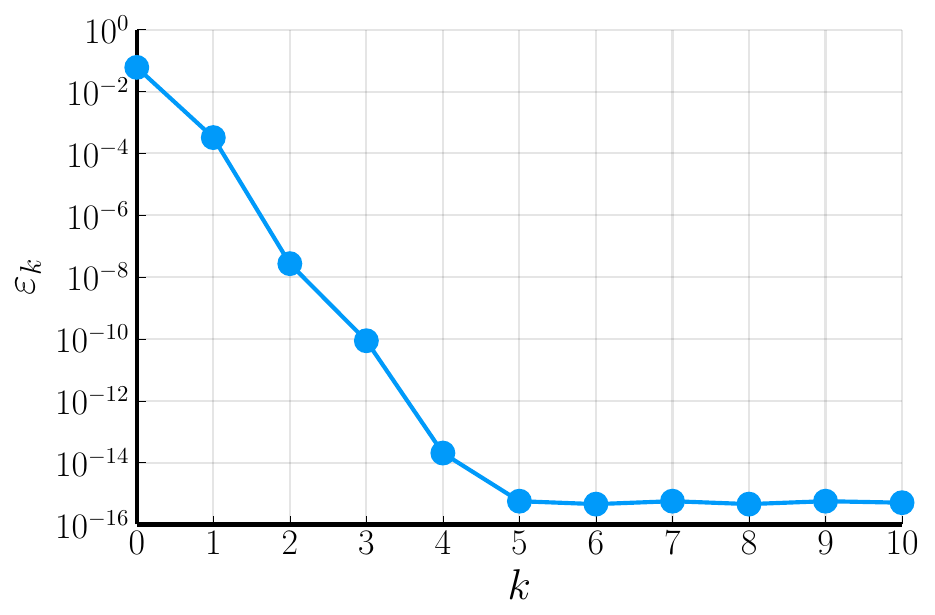}
        \caption{$N = 4$}
        % \label{}
    \end{subfigure}

    \caption{Linear convergence of \Cref{CP-AltLS} for $N$\textsuperscript{th}-order
    incoherently decomposable tensors.}
    \label{F:ideco}
\end{figure}

\subsection{Weights} \label{SS:exp-weights}

In \Cref{F:weights}, we plot the error in the \emph{squared} weights, computed as
$\norm{\vec{\lambda} \hd \vec{\lambda} - \vec[k]{\lambda} \hd \vec[k]{\lambda}}_\infty$, for
the tensors of \Cref{F:odeco-3,F:ideco-3}. 
For both tensors, the order of convergence of the weights matches that of the quantities
$\epsilon_k$, confirming \Cref{C:weights}.

\begin{figure}[htbp]
    \centering

    \begin{subfigure}{0.5\textwidth}
        \centering
        \includegraphics[width=\textwidth]{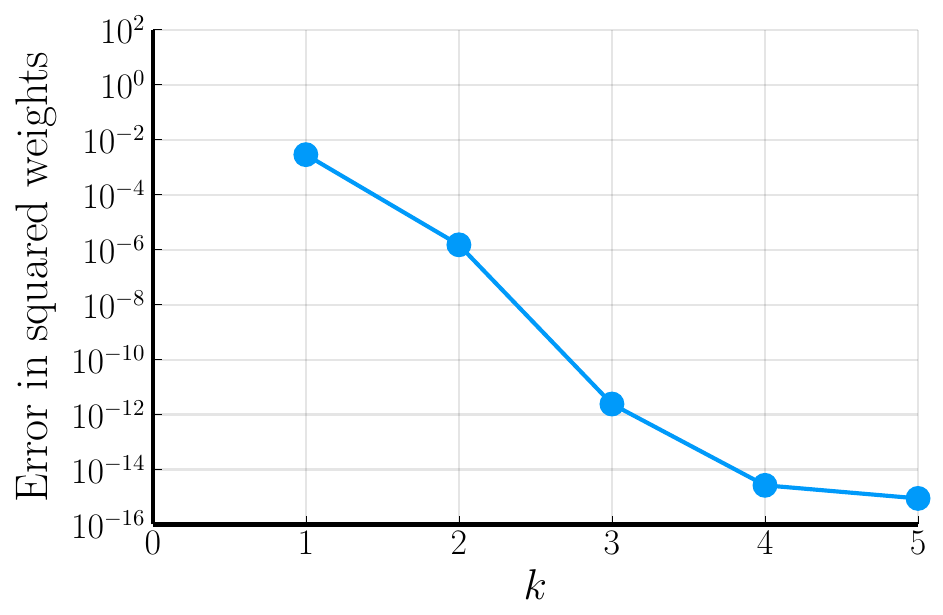}
        \caption{Orthogonally decomposable tensor}
        % \label{}
    \end{subfigure}%
    \begin{subfigure}{0.5\textwidth}
        \centering
        \includegraphics[width=\textwidth]{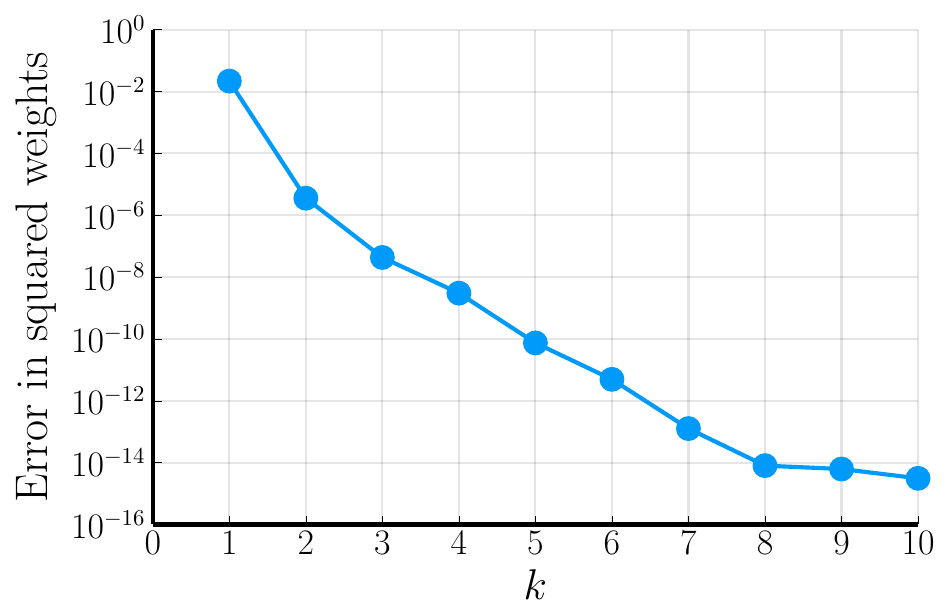}
        \caption{Incoherently decomposable tensor}
        % \label{}
    \end{subfigure}

    \caption{Convergence of weights in \Cref{CP-AltLS} for 3\textsuperscript{rd}-order
    orthogonally and incoherently decomposable tensors.}
    \label{F:weights}
\end{figure}

\subsection{Convergence acceleration of CP-SAltLS using SVD coherence reduction} \label{SS:exp-orth}

Given the rapid convergence of CP-SAltLS (and CP-PAltLS) for orthogonally decomposable
tensors, it is natural to attempt to accelerate the convergence of this algorithm for
general tensors by orthogonalizing the computed factor matrices.
The \emph{Orth-ALS} algorithm of Sharan and Valiant \cite{SharanValiant} does so by replacing
$\mat[n]{A}$ with $\mat[n]{Q}$, where $\mat[n]{Q} \mat[n]{R}$ is a QR factorization of the
(normalized) approximate factor matrix $\mat[n]{A}$.

However, this technique does not produce factor matrices suitable for the non-orthogonal but incoherent setting discussed in this paper. 
To address this, we propose incorporating a \emph{coherence reduction} procedure by replacing $\mat[n]{A}$
with $\mat[n]{U} (\mat[n]{\Sigma})^\omega (\mat[n]{V})^\tp$ for some parameter $\omega \in
[0, 1]$, where $\mat[n]{U} \mat[n]{\Sigma} (\mat[n]{V})^\tp$ represents an SVD of
$\mat[n]{A}$ prior to normalization. Note that $\omega = 1$ leaves each factor matrix $\mat[n]{A}$ unchanged, $\omega = 0$ is analogous to the orthogonal decomposition in the \emph{Orth-ALS} algorithm by Sharan and Valiant \cite{SharanValiant}, and $\omega \in (0,1)$ reduces the coherence without achieving orthogonality. In our experiment, we also observed that slightly more
accurate results could be obtained by performing the column-normalization step for all factor matrices at the end of each iteration
instead of for each individual factor matrix immediately. 
We found that performing regular AltLS iterations after these coherence-reduced AltLS
iterations -- by analogy with Sharan and Valiant's \emph{Hybrid-ALS} algorithm -- was often more
accurate than performing only regular AltLS iterations.

Specifically, \Cref{F:orth-random} displays the relative error $\frac{\norm{\ten{X} -
\ten[k]{X}}}{\norm{\ten{X}}}$ for
\[
    \ten{X} = 
    \vec{a}_1 \otr \vec{a}_2 \otr \vec{a}_3 + 
    \vec{a}_2 \otr \vec{a}_3 \otr \vec{a}_1 + 
    \vec{a}_3 \otr \vec{a}_1 \otr \vec{a}_2, 
\]
where $\vec{a}_1, \vec{a}_2, \vec{a}_3 \in \R^{10}$ have entries uniformly distributed in
$[0, 1)$ and the initial approximate factor matrices $\mat[n, 0]{A} \in \R^{10 \times 3}$ are
(normalized) random matrices. In this experiment, we performed 25 coherence-reduced iterations
followed by 25 regular iterations.

\Cref{F:orth-aminoacids} displays the same for Andersson and Bro's \emph{amino acids tensor}
$\ten{X} \in \R^{5 \times 51 \times 201}$ (consisting of fluorescence measurements of 3
amino acids from 5 samples, 51 excitations, and 201 emissions) \cite{Bro}, again with random
$\mat[n, 0]{A} \in \R^{I_n \times 3}$.  In this experiment, we performed 10 coherence-reduced
iterations followed by 20 regular iterations.

\begin{figure}[htbp]
    \centering

    \begin{subfigure}{0.5\textwidth}
        \centering
        \includegraphics[width=\textwidth]{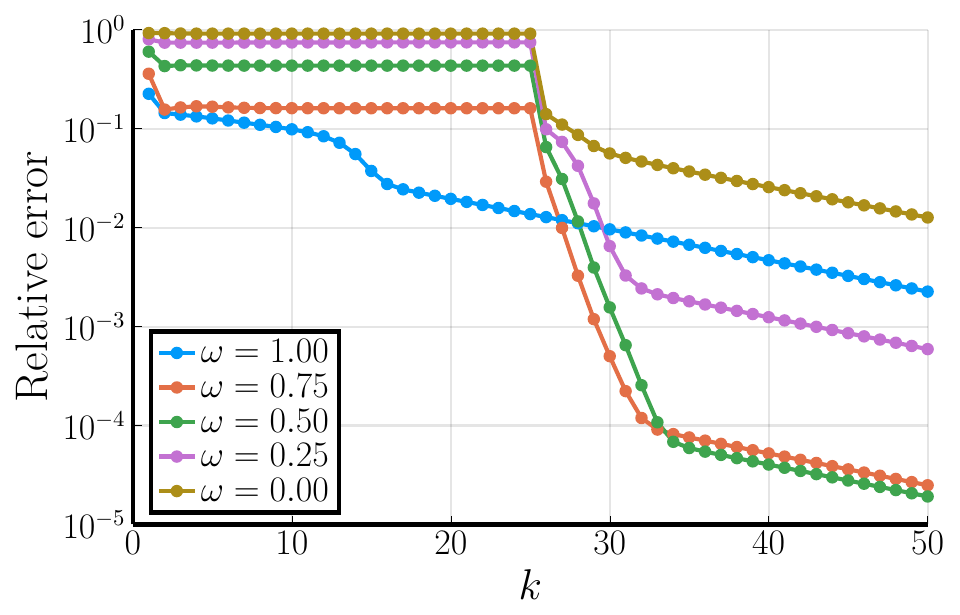}
        \caption{Random tensor}
        \label{F:orth-random}
    \end{subfigure}%
    \begin{subfigure}{0.5\textwidth}
        \centering
        \includegraphics[width=\textwidth]{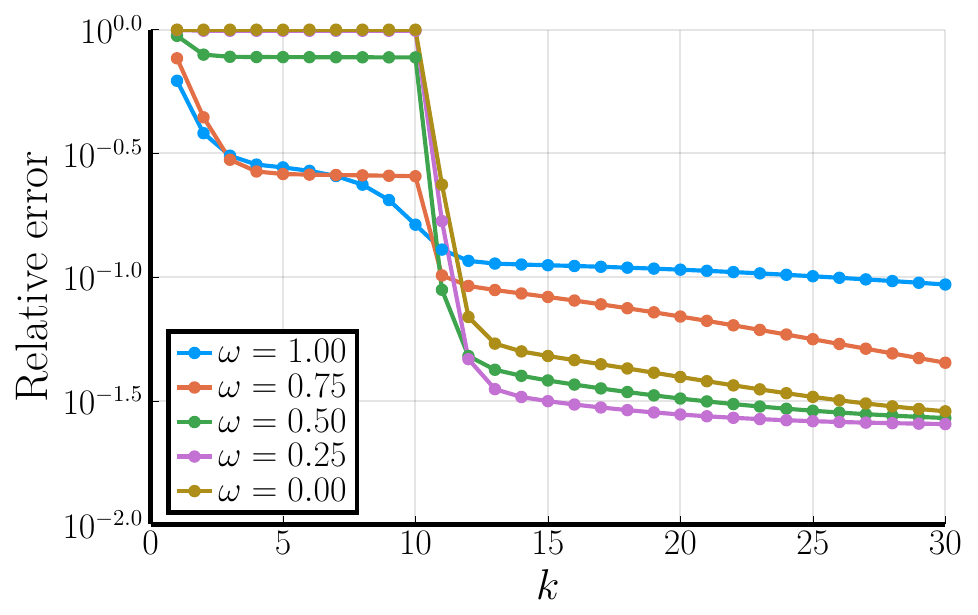}
        \caption{Amino acids tensor}
        \label{F:orth-aminoacids}
    \end{subfigure}

    \caption{Convergence of \Cref{CP-SAltLS} for coherence reduction parameter values ranging
    from $\omega = 1$ (no coherence reduction) to $\omega = 0$ (complete coherence reduction).}
    \label{F:orth}
\end{figure}

\section{Conclusion}

In this paper, we proved the first explicit quantitative local convergence theorems for the widely used CP-AltLS method, addressing a gap in its theoretical backing. We demonstrated that CP-AltLS converges polynomially with order $N-1$ for $N$\textsuperscript{th}-order orthogonally decomposable tensors and linearly for incoherently (nearly orthogonally) decomposable tensors. This convergence is dependent on the angles between the factors of the exact and approximate tensors. We believe our approach is more direct and less technical than the approaches in prior work and remains applicable to tensors of arbitrary rank, even allowing for factor matrices with small but nonzero mutual coherence.

We supported our theoretical analysis with numerical experiments and observed that the squared weights in the CP decomposition converge at the same rate as the factor matrices. We also explored accelerating the convergence of CP-AltLS by using an SVD-based coherence reduction scheme, and found that using coherence-reduced and then regular iterations often led to more accurate results than using only regular iterations. We believe our findings contribute theoretical insights into the behaviour of CP-AltLS and  suggest practical strategies for improving its performance as well as future theoretical directions.

\section{Acknowledgments}
M.~Iwen was partially supported by NSF DMS 2106472.
D.~Needell was partially supported by NSF DMS 2408912.
R.~Wang was supported by NSF CCF 2212065.  

% \clearpage
% \emergencystretch=1em
% \printbibliography
\bibliographystyle{siamplain}
\bibliography{main}

% \section*{Appendix}
% \renewcommand{\thetheorem}{\arabic{theorem}}

\appendix
\section{Technical lemmas} \label{appendix}

\begin{lemma} \label{L:innerprod-i}
    Suppose that $\mat{A}, \mat{B} \in \R^{m \times n}$ have normalized columns.  If
    $\inner{\vec{a}_j}{\vec{b}_j} \geq 0$ for all $j$ and $\epsilon \defeq \max_{j \in [n]}
    \, \abs{\sin \angle(\vec{a}_j, \vec{b}_j)}$, then
    \begin{enumerate}[label=\textup(\alph*\textup)]
        \item
            $\norm{\mat{B}^\tp \mat{A} - \mat{B}^\tp \mat{B}}_\mathrm{max} \leq \sqrt{2}
            \epsilon;$
        \item
            $\norm{\mat{A}^\tp \mat{A} - \mat{B}^\tp \mat{B}}_\mathrm{max} \leq 2\sqrt{2}
            \epsilon.$
    \end{enumerate}
\end{lemma}

\begin{proof}
    For part (a), observe that $\abs{\inner{\vec{a}_j}{\vec{b}_i} -
    \inner{\vec{b}_j}{\vec{b}_i}}^2 \leq \norm{\vec{a}_j - \vec{b}_j}^2 = 2(1 -
    \inner{\vec{a}_j}{\vec{b}_j})$, where $\inner{\vec{a}_j}{\vec{b}_j} = \cos
    \angle(\vec{a}_j, \vec{b}_j) \geq \sqrt{1 - \epsilon^2} \geq 1 - \epsilon^2$.
    Part (b) follows immediately from part (a) and the triangle inequality.
\end{proof}

\begin{lemma} \label{L:innerprod-o}
    Suppose that $\mat{A} \in \R^{m \times n}$ has normalized columns and 
    $\mat{B} \in \R^{m \times n}$ has orthonormal columns.
    If $\inner{\vec{a}_j}{\vec{b}_j} \geq 0$ for all $j$ and 
    $\epsilon \defeq \max_{j \in [n]} \, \abs{\sin \angle(\vec{a}_j, \vec{b}_j)}$, then
    $\norm{\D'(\mat{B}^\tp \mat{A})}_{1, 2} \leq \epsilon$.
    %, with equality if $\mat{B}^\tp \mat{A}$ has normalized columns.
\end{lemma}

\begin{proof}
    By Bessel's inequality,
    $\sum_{i \in [n]} \inner{\vec{a}_j}{\vec{b}_i}^2 \leq \norm{\vec{a}_j}^2 = 1$, so
    $(\sum_{i \neq j} \inner{\vec{a}_j}{\vec{b}_i}^2)^{\frac{1}{2}} \leq (1 -
    \inner{\vec{a}_j}{\vec{b}_j}^2)^{\frac{1}{2}} = \abs{\sin \angle(\vec{a}_j, \vec{b}_j)}$.
    % with equality if $\mat{B}^\tp \mat{A}$ has normalized columns.
\end{proof}

\begin{lemma} \label{L:inverse}
    Let $\mat{A}, \mat{B} \in \R^{n \times n}$ with $\D(\mat{A}) = \D(\mat{B}) = \mat{I}$.
    If $\norm{\D'(\mat{A} - \mat{B})}_\mathrm{max} \leq \epsilon$, $\norm{\D'(\mat{B})}_\mathrm{max} \leq
    \epsilon'$, and $(n-1)(\epsilon + \epsilon') < 1$, then $\mat{A}$ and $\mat{B}$ are invertible and 
    \begin{enumerate}[label=\textup(\alph*\textup)]
        \item
            \begin{align*}
                \norm{\mat{A}^{-1} - \mat{B}^{-1}}_\mathrm{max} 
                &\leq 
                \frac{\epsilon}{[1 - (n-1)(\epsilon + \epsilon')][1 - (n-1)\epsilon']} \\
                &\leq
                \frac{\epsilon}{[1 - (n-1)(\epsilon + \epsilon')]^2};
            \end{align*}
        \item
            \[
                \norm{\mat{B}^{-1}}_\mathrm{max} \leq 
                \frac{1}{1 - (n-1)\epsilon'};
            \]
        \item
            \[
                \norm{\D'(\mat{B}^{-1})}_\mathrm{max} \leq 
                \frac{(n-1)\epsilon'}{1 - (n-1)\epsilon'}.
            \]
    \end{enumerate}
\end{lemma}

\begin{corollary} \label{C:inverse}
    Let $\mat{A} \in \R^{n \times n}$ with $\D(\mat{A}) = \mat{I}$. If
    $\norm{\D'(\mat{A})}_\mathrm{max} \leq \epsilon$ and $(n-1)\epsilon < 1$, then $\mat{A}$
    is invertible and
    \[
        \norm{\mat{A}^{-1} - \mat{I}}_\mathrm{max} \leq \frac{\epsilon}{1 - (n-1)\epsilon}.
    \]
\end{corollary}

\begin{proof}
    We first note that $\mat{A}^{-1} = [\mat{I} - (-\D'(\mat{A}))]^{-1} = \sum_{k=0}^\infty
    (-\D'(\mat{A}))^k$, where the series converges since $\norm{\D'(\mat{A})}_1 \leq (n-1)
    \norm{\D'(\mat{A})}_\mathrm{max} \leq (n-1) (\epsilon + \epsilon') < 1$ (and since the
    1-norm is submultiplicative), and similarly for $\mat{B}$. Then, since
    $\norm{\D'(\mat{A})}_\infty \leq (n-1)(\epsilon + \epsilon')$ and $\norm{\D'(\mat{B})}_1
    \leq (n-1)\epsilon'$, we have
    \begin{align*}
        \norm{\mat{A}^{-1} - \mat{B}^{-1}}_\mathrm{max}
        &= \norm{\mat{A}^{-1} (\mat{B} - \mat{A}) \mat{B}^{-1}}_\mathrm{max} \\
        &\leq \norm{\mat{A}^{-1}}_\infty \norm{\mat{B} - \mat{A}}_\mathrm{max}
        \norm{\mat{B}^{-1}}_1 \\
        &\leq \frac{1}{1 - (n-1)(\epsilon + \epsilon')} \cdot \epsilon \cdot
        \frac{1}{1 - (n-1) \epsilon'}.
    \end{align*}
    The other two inequalities are consequences of the series expansion $\mat{B}^{-1} = \mat{I} +
    \sum_{k=1}^\infty (-\D'(\mat{B}))^k$.
\end{proof}

\begin{lemma} \label{L:product-o}
    Let $\mat{A}, \mat{B} \in \R^{n \times n}$. If $\norm{\D(\mat{A})}_\mathrm{max} \leq 1$,
    $\norm{\D(\mat{A}) - \mat{I}}_\mathrm{max} \leq \epsilon_\mathrm{A} \leq 1$,
    $\norm{\D'(\mat{A})}_\mathrm{1, 2} \leq \epsilon_\mathrm{A}'$, and $\norm{\mat{B} -
    \mat{I}}_\mathrm{max} \leq \epsilon_\mathrm{B} \leq 1$, then
    \begin{enumerate}[label=\textup(\alph*\textup)]
        \item
            $\norm{\D(\mat{A}\mat{B}) - \mat{I}}_\mathrm{max}
            \leq
            1 - (1 - \epsilon_\mathrm{A})(1 - \epsilon_\mathrm{B}) + 
            (n-1) \epsilon_\mathrm{A}' \epsilon_\mathrm{B}^{};$
        \item
            $\norm{\D'(\mat{A}\mat{B})}_\mathrm{1, 2}
            \leq
            (n-1)^{\frac{1}{2}} \epsilon_\mathrm{B}^{} + 
            \epsilon_\mathrm{A}' (1 + \epsilon_\mathrm{B}^{}) +
            (n-1) \epsilon_\mathrm{A}' \epsilon_\mathrm{B}^{}.$
    \end{enumerate}
\end{lemma}

\begin{proof}
    Observe that 
    \begin{align*}
        \mat{AB} 
        &= (\D(\mat{A}) + \D'(\mat{A}))(\D(\mat{B}) + \D'(\mat{B})) \\
        &= \D(\mat{A})\D(\mat{B}) + \D(\mat{A})\D'(\mat{B}) + 
        \D'(\mat{A})\D(\mat{B}) + \D'(\mat{A})\D'(\mat{B}), 
    \end{align*}
    where the first term is diagonal and the second and third
    terms have zero diagonal parts. As a result, 
    \begin{align*} 
        \D(\mat{AB} - \mat{I}) 
        &= \D(\D(\mat{A})\D(\mat{B}) - \mat{I}) + \D(\D'(\mat{A})\D'(\mat{B})), \\ 
        \D'(\mat{AB}) 
        &= \D'(\D(\mat{A})\D'(\mat{B})) + \D'(\D'(\mat{A})\D(\mat{B})) + \D'(\D'(\mat{A})\D'(\mat{B})).
    \end{align*} 
    Bounding each of these terms separately yields the results above.
\end{proof}

\begin{lemma} \label{L:normalization-o}
    Suppose that $\mat{A} = \mat{B} \mat{V}$ for some $\mat{B} \in \R^{n \times n}$
    with orthonormal columns and $\mat{V} \in \R^{n \times n}$.
    If $\norm{\D(\mat{\Lambda}^{-1} \mat{V}) -
    \mat{I}}_\mathrm{max} \leq \epsilon < 1$ and $\norm{\D'(\mat{\Lambda}^{-1}
    \mat{V})}_{1, 2} \leq \epsilon'$ for some diagonal matrix $\mat{\Lambda} =
    \diag(\lambda_1, \dots, \lambda_n)$, then
    \[
        \max_{j \in [n]} \, \abs{\sin \angle(\vec{a}_j, \vec{b}_j)}
        \leq 
        \frac{\kappa \epsilon'}{1 - \epsilon},
    \]
    where $\kappa \defeq \max_{j \in [n]} \abs{\lambda_j} / \min_{j \in [n]}
    \abs{\lambda_j}$.
\end{lemma}

\begin{proof}
    Let $\mat{W} = \mat{\Lambda}^{-1} \mat{V}$ so that $\mat{A} = \mat{B} \mat{\Lambda}
    \mat{W}$ with $\norm{\D(\mat{W}) - \mat{I}}_\mathrm{max} \leq \epsilon$ and
    $\norm{\D'(\mat{W})}_{1, 2} \leq \epsilon'$. Then $\vec{a}_j = \sum_{i \in [n]}
    \lambda_i w_{ij} \vec{b}_i$, so $\norm{\vec{a}_j}^2 = \sum_{i} \lambda_i^2
    w_{ij}^2$ and
    $\inner{\vec{a}_j}{\vec{b}_j}^2 = \lambda_j^2 w_{jj}^2$.
    Hence $\norm{\vec{a}_j}^2 - \inner{\vec{a}_j}{\vec{b}_j}^2 = 
    \sum_{i \neq j} \lambda_i^2 w_{ij}^2$.
    On the other hand, $\norm{\vec{a}_j}^2 \geq \lambda_j^2 w_{jj}^2 \geq 
    \lambda_j^2 (1-\epsilon)^2$.
    As a result,
    \begin{align*}
        \sin^2 \angle(\vec{a}_j, \vec{b}_j)
        &= \frac{\norm{\vec{a}_j}^2 - \inner{\vec{a}_j}{\vec{b}_j}^2}{\norm{\vec{a}_j}^2}
        \leq \frac{\kappa^2 (\epsilon')^2}{(1-\epsilon)^2}.
        \qedhere
    \end{align*}
\end{proof}

\begin{lemma} \label{L:product-i}
    Let $\mat{A}, \mat{\tilde{A}}, \mat{B}, \mat{\tilde{B}} \in \R^{n \times n}$. 
    If $\norm{\D(\mat{\tilde{A}} - \mat{A})}_\mathrm{max} \leq \epsilon_\mathrm{A}$ and
    $\norm{\D'(\mat{\tilde{A}} - \mat{A})}_\mathrm{max} \leq \epsilon_\mathrm{A}'$, and
    likewise for $\mat{B}$ and $\mat{\tilde{B}}$, then
    \begin{enumerate}[label=\textup(\alph*\textup)]
        \item
            \begin{align*}
                \norm{\D(\mat{\tilde{A}}\mat{\tilde{B}} - \mat{A}\mat{B})}_\mathrm{max}
                &\leq
                (n-1) \norm{\D'(\mat{\tilde{A}})}_\mathrm{max} \cdot \epsilon_\mathrm{B}' + 
                \norm{\D(\mat{\tilde{A}})}_\mathrm{max} \cdot \epsilon_\mathrm{B}^{} + {} \\
                &\hphantom{{}\leq{}}
                (n-1) \norm{\D'(\mat{B})}_\mathrm{max} \cdot \epsilon_\mathrm{A}' + 
                \norm{\D(\mat{B})}_\mathrm{max} \cdot \epsilon_\mathrm{A}^{}.
            \end{align*}
        \item
            \begin{align*}
                \norm{\D'(\mat{\tilde{A}}\mat{\tilde{B}} - \mat{A}\mat{B})}_\mathrm{max}
                &\leq
                (n-1) \norm{\mat{\tilde{A}}}_\mathrm{max} \cdot \epsilon_\mathrm{B}' + 
                \norm{\D'(\mat{\tilde{A}})}_\mathrm{max} \cdot \epsilon_\mathrm{B}^{} + {} \\
                &\hphantom{{}\leq{}}
                (n-1) \norm{\mat{B}}_\mathrm{max} \cdot \epsilon_\mathrm{A}' + 
                \norm{\D'(\mat{B})}_\mathrm{max} \cdot \epsilon_\mathrm{A}^{}.
            \end{align*}
    \end{enumerate}
\end{lemma}

\begin{proof}
    Write $\mat{\tilde{A}} \mat{\tilde{B}} - \mat{A} \mat{B} =
    \mat{\tilde{A}}(\mat{\tilde{B}} - \mat{B}) + (\mat{\tilde{A}} - \mat{A}) \mat{B}$ and
    apply the triangle inequality to each term.
\end{proof}

\begin{lemma} \label{L:normalization-i}
    Suppose that $\mat{A} = \mat{B} \mat{V}$ for some $\mat{B} \in \R^{n \times n}$
    with normalized columns and $\mat{V} \in \R^{n \times n}$.
    If $\norm{\D(\mat{\Lambda}^{-1} \mat{V}) -
    \mat{I}}_\mathrm{max} \leq \epsilon < 1$ and $\norm{\D'(\mat{\Lambda}^{-1}
    \mat{V})}_\mathrm{max} \leq \epsilon'$ for some diagonal matrix $\mat{\Lambda} =
    \diag(\lambda_1, \dots, \lambda_n)$, and $4(n-1)\kappa \epsilon' +
    [(n-1)\kappa\epsilon']^2 < (1-\epsilon)^2$, then
    \[
        \max_{j \in [n]} \, \abs{\sin \angle(\vec{a}_j, \vec{b}_j)}
        \leq 
        \left(\frac{2}{(1 - \epsilon)^2 - 4(n-1)\kappa \epsilon' -
        [(n-1)\kappa\epsilon']^2}\right)^\frac{1}{2} (n-1) \kappa \epsilon',
    \]
    where $\kappa \defeq \max_{j \in [n]} \abs{\lambda_j} / \min_{j \in [n]}
    \abs{\lambda_j}$.
\end{lemma}

\begin{proof}
    Let $\mat{W} = \mat{\Lambda}^{-1} \mat{V}$ so that $\mat{A} = \mat{B} \mat{\Lambda}
    \mat{W}$ with $\norm{\D(\mat{W}) - \mat{I}}_\mathrm{max} \leq \epsilon$ and
    $\norm{\D'(\mat{W})}_\mathrm{max} \leq \epsilon'$. Then $\vec{a}_j = \sum_{i \in [n]}
    \lambda_i w_{ij} \vec{b}_i$, so 
    \begin{align*}
        \norm{\vec{a}_j}^2 
        &= \textstyle \sum_{i, i'} \lambda_i \lambda_{i'} w_{ij} w_{i'j} \inner{\vec{b}_i}{\vec{b}_{i'}}, \\
        \inner{\vec{a}_j}{\vec{b}_j}^2 
        &= \textstyle \sum_{i, i'} \lambda_i \lambda_{i'} w_{ij} w_{i'j}
        \inner{\vec{b}_i}{\vec{b}_j} \inner{\vec{b}_{i'}}{\vec{b}_j}.
    \end{align*}
    Hence
    \begin{align*}
        \norm{\vec{a}_j}^2 - \inner{\vec{a}_j}{\vec{b}_j}^2 
        &= \textstyle \sum_{i, i' \neq j} \lambda_i \lambda_{i'} w_{ij} w_{i'j} (\inner{\vec{b}_i}{\vec{b}_{i'}} -
    \inner{\vec{b}_i}{\vec{b}_j} \inner{\vec{b}_{i'}}{\vec{b}_j}) \\
        &\leq \textstyle 2(\epsilon')^2 \sum_{i, i' \neq j} \abs{\lambda_i} \abs{\lambda_{i'}}
    \end{align*}
    since $\mat{B}$ has normalized columns. On the other hand,
    \begin{align*}
        \norm{\vec{a}_j}^2
        &=
        \textstyle
        \lambda_j^2 w_{jj}^2 + 2 \sum_{i \neq j} \lambda_i \lambda_j w_{ij} w_{jj}
        \inner{\vec{b}_i}{\vec{b}_j} + \sum_{i, i' \neq j} \lambda_i \lambda_{i'} w_{ij}
        w_{i'j} \inner{\vec{b}_i}{\vec{b}_{i'}} \\
        &\geq
        \textstyle
        \lambda_j^2 (1-\epsilon)^2 - 4\epsilon' \abs{\lambda_j} \sum_{i \neq j} \abs{\lambda_i}
        - (\epsilon')^2 \sum_{i, i' \neq j} \abs{\lambda_i} \abs{\lambda_{i'}}.
    \end{align*}
    As a result,
    \begin{align*}
        \sin^2 \angle(\vec{a}_j, \vec{b}_j)
        &= \frac{\norm{\vec{a}_j}^2 - \inner{\vec{a}_j}{\vec{b}_j}^2}{\norm{\vec{a}_j}^2} \\
        &\leq \frac{2(\epsilon')^2 \cdot (n-1)^2\kappa^2}{(1-\epsilon)^2 - 4\epsilon' \cdot
        (n-1) \kappa - (\epsilon')^2 \cdot (n-1)^2 \kappa^2}.
        \qedhere
    \end{align*}
\end{proof}

\end{document}